\documentclass[11pt,onecolumn,a4paper,oneside]{article}
\usepackage[utf8]{inputenc}
\usepackage[ngerman, english]{babel}
\usepackage[T1]{fontenc}
\usepackage{amsmath}
\usepackage{amssymb}
\usepackage{verbatim}

\usepackage[arrow, matrix, curve]{xy}
\usepackage{latexsym}
\usepackage{extarrows}
\usepackage{amsthm}
\usepackage{hyperref}
\usepackage{graphicx}

\newcommand{\norm}[1]{\left\lVert#1\right\rVert}
\newcommand{\normd}{\left\lVert\cdot\right\rVert}
\newcommand{\val}[1]{\left\lvert#1\right\rvert}
\newcommand{\vald}{\left\lvert\cdot\right\rvert}
\newcommand{\stalk}{\mathcal O_{X,x}}
\newcommand{\gar}{\gamma_0}
\DeclareMathOperator{\Spa}{Spa}

\DeclareMathOperator{\im}{Im}
\DeclareMathOperator{\id}{id}
\DeclareMathOperator{\Hom}{Hom}
\DeclareMathOperator{\supp}{supp}
\DeclareMathOperator{\sep}{sep}
\DeclareMathOperator{\ev}{ev}
\DeclareMathOperator{\QCOuv}{QCOuv}
\DeclareMathOperator{\Spec}{Spec}
\newcommand{\mt}[1]{\mathcal#1}
\theoremstyle{plain}
\newtheorem{prop}{Proposition}[section]
\newtheorem{theo}[prop]{Theorem}

\newtheorem{lem}[prop]{Lemma}

\newtheorem{defn}[prop]{Definition}
\newtheorem{rem}[prop]{Remark}

\newtheorem{exmp}[prop]{Example}
\newtheorem{const}[prop]{Construction}

\begin{document}









\begin{center}
\Large  A comparison of adic spaces and Berkovich spaces \\
\vspace{0.5cm}
\normalsize Timo Henkel
\footnote{Timo Henkel, Department of Mathematics, AG-Algebra, TU Darmstadt, Schloßgartenstraße 7, 64289 Darmstadt, Germany; thenkel@mathematik.tu-darmstadt.de}
\vspace{0.2cm}
\end{center} 
\noindent \textbf{Abstract}. This paper reviews the equivalence between the category of taut adic spaces that are locally of finite type and the category of strictly analytic Berkovich spaces. An explicit construction of this functor is provided by using the terminology of valuative spaces. 

\section*{Introduction}

After the discovery of $p$-adic numbers in the late 19th century one goal was to establish a theory of analytic functions for those \emph{non-archimedean fields} similarly to that over archimedean fields like $\mathbb{R}$ or $\mathbb{C}$. Since the valuation of a non-archimedean field $k$ satisfies the ultrametric triangle inequality, it is totally disconnected. This fact leads to strange occurrences when defining analytic spaces over $k$ in a naive way. It turned out that there were too many possible \emph{coverings} so that functions, which should not be analytic, had this property. The great idea of Tate was to restrict the setting to so called \emph{admissible coverings} which led to the notion of a \emph{Grothendieck topology} and the definition of rigid analytic spaces. \\
In the last decades more general approaches were developed, taking the theory into different directions. In particular, this paper focuses on the connection between Huber's theory of adic spaces and Berkovich's theory of analytic spaces which both extend the classic theory. In the affine case, Huber's theory works with certain pairs of general topological rings. In contrast to this, the fundamental rings considered in Berkovich's theory are generalized Tate algebras that admit scaling, which was not possible in the classic theory. \\
It is easy to see that those objects coincide if we assume both rings to satisfy certain finiteness conditions over $k$. Straightforward, one obtains an equivalence of categories for the affine objects in both categories (cf. §4). Nevertheless, the obtained spaces in both theories differ tremendously, already in view of topological aspects. The main problem, when one tries to generalize the equivalence to non-affinoid adic spaces and Berkovich spaces, are the different glueing procedures that create global objects out of affine ones. In the adic setting spaces are glued by open immersions and in the Berkovich case, we have affinoid domain embeddings which are closed immersions from a topological point of view. \\
The general equivalence of categories, which is also the main result of this paper (cf. \ref{theo. final theorem}), was previously proven by using the category of finite type rigid analytic spaces as a bridge in between (cf. \cite{Hu3} Proposition 4.5 and \cite{Be3} Theorem 1.6.1). Unfortunately, this proof does not give much information on how the corresponding spaces are connected. Therefore we provide a more direct proof following Fujiwara and Kato in \cite{FK} in which we explicitly construct a Berkovich space structure on the Hausdorff quotient of a taut adic space. Some of the work will be done in the setting of valuative spaces, as introduced in \cite{FK}. \\

As mentioned above, the main theorem of this paper is the following:

\begin{theo}\label{theo. final theorem oben}(cf. \ref{theo. final theorem})
There is an equivalence of categories:
\begin{gather*}		
\lbrace \text{taut adic spaces that are locally of finite type over }k \rbrace \\		
\cong \\
\lbrace \text{Hausdorff strictly} \ k\text{-analytic Berkovich spaces} \rbrace
\end{gather*}
sending $(X,\mt O_X, (v_x)_{x\in X})$ to $([X],\mt A, \tau)$, where $[X]$ is the universal Hausdorff quotient of $X$.
\end{theo}

The paper is divided into six sections. 
The first paragraph recalls some basic definitions of topological spaces that are used for the later content. 
In the second section a short introduction to Huber's theory of adic spaces is given in which we focus on those objects satisfying certain finiteness conditions over $k$. 
The third section shortly establishes analytic spaces in the sense they were defined by Berkovich. Similarly to the section before, we stress the properties which are important to compare particular Berkovich spaces with adic spaces. 
The affinoid case of this comparison will be done in paragraph four.
The fifth section introduces the notion of valuative spaces and their separated quotients and explains the relation of those terms to a possible comparison between the theories mentioned before.
The main result Theorem \ref{theo. final theorem} of this work is proven in the last paragraph by an explicit construction of the functor from the category of taut adic spaces that are locally of finite type over $k$ to the category of Hausdorff strictly $k$-analytic Berkovich spaces.\\

\noindent \textbf{Acknowledgements}.\\ I would like to thank Torsten Wedhorn who was my supervisor for this paper which is my master thesis. I am grateful to him for suggesting this topic and for the patience he had during our countless discussions.  Moreover, I am thankful to the members of the AG-Algebra at TU Darmstadt who warmly hosted me during the finial period of my work for this paper.\\
\newpage
\noindent Notations and assumptions:
\begin{enumerate}
\item Let all rings be commutative and with unit. 
\item All complete spaces are Hausdorff by definition.
\item For a ring $A$ with subsets $R$ and $S$, the set $R\cdot S$ denotes the additive subgroup of $A$ generated by $\lbrace rs \ | \ r \in R, \ s \in S \rbrace$.
\item Throughout this paper we fix a non-archimedean field $(k,\vald)$, i.e. a topological field whose topology is complete and induced by a non-trivial valuation $\vald$. In other words, $\vald$ is a map $k \rightarrow \mathbb{R}^{\geq 0}$ satisfying
\begin{enumerate}
\item $\val 0 =0$ and $\val 1 =1$,
\item $\val{ab}=\val a \val b$ and
\item $\val{a+b} \leq \max\lbrace\val a, \val b\rbrace$
\end{enumerate}
 for all $a,b \in k$.
 In particular, $k$ contains topologically nilpotent units which play an important role in the whole theory.

 Note that in this case $\sqrt{\val{k^\times}}:= \lbrace a \in \mathbb{R}^{\geq 0} \ | \ \exists \ n \in \mathbb{N} \ \text{with} \ a^n \in \val{k^\times} \rbrace$ is dense in $\mathbb{R}^{\geq 0}$. Indeed: Let $a\in \mathbb{R}^{\geq 0}$, $a<1$ and $\epsilon > 0$. Choose $c \in k$ such that $\val c < 1$ and take $n \in \mathbb{N}$ with $1-\sqrt[n]{\val c}< \epsilon$. Let $k\in \mathbb{N}$ such that $\sqrt[n]{\val c}^k \leq a \leq \sqrt[n]{\val c}^{k-1}$. Then $a-\sqrt[n]{\val {c^k}} < \epsilon$. Since inversion is continuous in $\mathbb{R}$, the result follows for all $a>1$ as well. 

\end{enumerate}

\tableofcontents

\section{General topology}

We start with some basic notions from general topology which will play an important role later in this work.

\begin{rem} 
\emph{
For a topological space $X$, we say that $\tilde{x}\in X$ is a \emph{generization} of $x \in X$, if $x \in \overline{\lbrace \tilde x \rbrace}$. Which means that $\tilde{x}$ is contained in any open neighbourhood of $x$.
By $G_x$ we denote the set of all generizations of $x$, or alternatively define $G_x$ to be the intersection of all open neighbourhoods of $x$. Now assume $X$ to be a $T_0$-space (i.e. for two distinct elements there exists an open set containing one of them but not the other).\\
Then the relation on $X$
\begin{gather*}		
x \leq y \ \text{if and only if} \ y \ \text{is a generization of}\ x
\end{gather*}
is a partial order on $X$ where we need the $T_0$-property to assure antisymmetry. $x \in X$ is said to be \emph{maximal} (resp. \emph{minimal}) if $x$ is maximal (resp. minimal) with respect to this ordering. Note that $x$ is the unique closed point of $G_x$ and that $x \in X$ is minimal if and only if $\lbrace x \rbrace$ is closed in $X$.  
}
\end{rem}

For certain spaces of interest, we will see that the sets $G_x$ are totally ordered. This leads to interesting observations which will we study in §6.

\begin{defn}\label{def. sober quasi sepatated...}
\emph{Let $X$ be a topological space.
\begin{enumerate}
\item We say that $X$ is \emph{compact} if it is quasi-compact and Hausdorff.
\item $X$ is said to be \emph{locally Hausdorff} if every point of $X$ has an open neighbourhood that is a Hausdorff space under the subspace topology.
\item If every point of $X$ has a compact neighbourhood contained in an open Hausdorff neighbourhood, then we say that $X$ is \emph{locally compact}. 
\item $X$ is called \emph{sober} if every closed irreducible subset of $X$ has a unique generic point.
\item We say that $X$ is \emph{quasi-separated} if the intersection of two open, quasi-compact subsets of $X$ again is quasi-compact.
\item $X$ is said to be \emph{taut} if it is quasi-separated and the closure of an open quasi-compact subset is again quasi-compact.
\item $X$ is called \emph{coherent} if it satisfies the following conditions:
	\begin{enumerate}
	\item $X$ has a basis of its topology which consists of quasi-compact subsets;
	\item $X$ is quasi-compact and quasi-separated.
	\end{enumerate}
\item If $X$ is coherent, we denote by $\QCOuv(X)$ the set of all open and quasi-compact subsets of $X$.
\item $X$ is said to be \emph{locally coherent} if it admits an open covering of coherent subspaces.
 \end{enumerate}
}
\end{defn}

\begin{rem}\label{rem. connection to spectral }
\emph{Some authors call topological spaces, that are coherent and sober, \emph{spectral}. This naming is justified by the fact that a topological space is spectral if and only if it is homeomorphic to the prime spectrum of a ring (cf. \cite{Ho}). We will stick with the notation in \cite{FK} and primarily use the notion of coherent sober spaces.
}
\end{rem}

\begin{defn}\label{def. locally quasi-compact}
\emph{Let $X$ be a topological space.
\begin{enumerate}
\item A continuous map $f:X \rightarrow Y$ between topological spaces is called \emph{quasi-compact} if $f^{-1}(V)$ is a quasi-compact subset of $X$ for any quasi-compact open subset $V \subseteq Y$.
\item We say that a subset $U\subseteq X$ is \emph{retro-compact} if the inclusion map $U \hookrightarrow X$ is quasi-compact.
\item A continuous map $f:X \rightarrow Y$ between locally coherent spaces is said to be \emph{locally quasi-compact} if the map $f|_U:U\rightarrow V$ is quasi-compact for any pair $(U,V)$ of coherent open subsets $U \subseteq X$ and $V \subseteq Y$ such that $f(U) \subseteq V$.
\end{enumerate}
}
\end{defn}

\begin{prop}\label{prop: quasi-compact map between locally coherent spaces}
Let $f:X \rightarrow Y$ be a locally quasi-compact map of locally coherent spaces and let $V \subseteq Y$ be a retrocompact open subset. Then $f^{-1}(V)$ is retrocompact as well.
\end{prop}
\begin{proof}
\cite{FK} Proposition 0.2.2.25
\end{proof}

\section{Adic spaces}

In this section we recall the definition and certain properties of adic spaces which were originally introduced by Huber in \cite{Hu2}. The discussion mainly follows \cite{We2} and uses results from \cite{Hu2} and \cite{Hu3}. 

\subsection{Huber rings and Huber pairs}
The fundamental algebraic objects in the theory of adic spaces are certain pairs of topological rings which are in the focus of this subsection.
\begin{defn}\label{def. Huber ring / Tate ring}
\emph{Let $A$ be a topological ring.
\begin{enumerate}
\item $A$ is called \emph{adic} if it possesses an ideal $I$ such that the family $(I^n)_{n\in \mathbb{N}}$ forms a fundamental system of open neighbourhoods of $0$ in $A$. In this case $I$ is said to be an \emph{ideal of definition of $A$}.
\item $A$ is called a \emph{Huber ring} if $A$ possesses an open adic subring $A_0$ with finitely generated ideal of definition $I$. We call $A_0$ a \emph{ring of definition of $A$} and $(A_0,I)$ a \emph{pair of definition of $A$}.
\item A subset $S\subseteq A$ is said to be \emph{bounded} if for every neighbourhood $U$ of $0$ in $A$ there exists a neighbourhood $V$ of $0$ in $A$ such that $V\cdot S \subseteq U$ (here we assume $V\cdot U$ to be $\lbrace vu \ | \ v\in V, \ u \in U \rbrace$).
\item $x \in A$ is said to be \emph{power-bounded} (resp.  \emph{topologically nilpotent}) if $\lbrace x^n \ | \ n\in \mathbb{N} \rbrace$ is bounded in $A$ (resp. if $\lim_{n\rightarrow \infty} x^n =0$). 
\end{enumerate}
We define:\begin{center}
 $A^o:= \lbrace x \in A \ | \ \text{x is power-bounded} \rbrace$,\\
 $A^{oo}:= \lbrace x \in A \ | \ \text{x is topologically nilpotent} \rbrace$.
\end{center}
\begin{enumerate}
\item[v)]$A$ is called a \emph{Tate ring} if it is a Huber ring and has a topologically nilpotent unit.
\end{enumerate}
}
\end{defn}

\begin{rem}\label{rem. power bounded elemnts form ring}
\emph{
Note that if $A$ is an adic ring, then $A^o$ is a subring of $A$, which contains $A^{oo}$ as an ideal.
}
\end{rem}

\begin{defn}\label{def. Huber pair}
\emph{
\begin{enumerate}
\item
A \emph{Huber pair} is a pair $(A,A^+)$ where $A$ is a Huber ring and $A^+$ is an integrally closed open subring of $A$ which is contained in $A^o$. We call such a ring $A^+$ \emph{ring of integral elements of $A$}.
\item A Huber pair $(A,A^+)$ is said to be \emph{complete} if $A$ is a complete topological ring with respect to its uniform structure induced by its topology.
\item A \emph{morphism of Huber pairs} $\varphi:(A,A^+)\rightarrow (B,B^+)$ is a continuous ring homomorphism $\varphi: A\rightarrow B$ such that $\varphi(A^+)\subseteq B^+$.
\end{enumerate}
}
\end{defn}

\begin{rem}\label{rem. complete Huber pairs}
\emph{
Let $(A,A^+)$ be a Huber pair. Let $\hat{A}$ denote the completion of $A$ as an additive topological group. Since the ring multiplication can be extended to $\hat{A}$, we obtain a complete topological ring which still is a Huber ring. If we denote the integral closure of the topological closure of $A^+$ in $\hat{A}$ by $\hat{A}^+$, we obtain a Huber pair $(\hat{A}, \hat{A}^+)$ which we call the \emph{completion} of $(A,A^+)$.
}
\end{rem}

The following construction leads to the definition of morphisms which are topologically of finite type and hence to Huber pairs that are topologically of finite type over $k$. Those pairs of rings will be the fundamental algebraic objects for our comparison.

\begin{defn}\label{defn. restricted power series}
\emph{
Let $(A,A^+)$ be a complete Huber pair and $T_1,...,T_n$ finite subsets of $A$ such that $T_i \cdot A$ is open for each $i=1,...,n$. Consider the ring of formal power series of $A$ denoted by $A [[ X ]]= A [[X_1,...,X_n ]]$. Using the usual multi-index notation, we define
\begin{gather*}		
A\langle X \rangle _T := \Bigl\lbrace \sum_{\nu \in \mathbb{N}_0^n} a_\nu X^\nu \in A[[X]] \ ;  \  \begin{tabular}{lcr}
    $a_\nu \in T^\nu \cdot U \ \text{for all open subgroups}$ \\
  $\ U \subseteq A \ \text{and almost all} \ \nu \in \mathbb{N}_0^n$ \\
 \end{tabular} \Bigr\rbrace.
\end{gather*} 
We endow $A\langle X \rangle_T$ with the unique ring topology such that a fundamental system of open neighbourhoods of $0$ in $A\langle X \rangle_T$ is given by sets of the form
\begin{gather*}		
   U_{\langle X \rangle} := \Bigl\lbrace  \sum_{\nu \in \mathbb{N}_0^n} a_\nu X^\nu \in A[[X]] \ | \ a_\nu \in T^\nu \cdot U \ \text{for all}  \ \nu \in \mathbb{N}^n_0 \Bigr\rbrace,
\end{gather*}   
 where $U$ is an open subgroup of A.
This makes $A\langle X \rangle _T$ a complete Huber ring.
If $T_i=\lbrace 1 \rbrace$ for all $i=1,..,n$, we simply write $A\langle X \rangle = A\langle X_1,...,X_n \rangle$.
Moreover, we define $A\langle X \rangle ^+_T$ to be the integral closure of 
\begin{gather*}		
\Bigl\lbrace \sum_{\nu \in \mathbb{N}_0^n} a_\nu X^\nu \in A[[X]] \ | \ a_\nu \in T^\nu \cdot A^+ \ \text{for all} \ \nu \in \mathbb{N}_o^n \Bigr\rbrace
\end{gather*}
in $A\langle X \rangle_T$.
So we get a Huber pair 
\begin{gather*}		
A\langle X_1,...,X_n \rangle _{T_1,...,T_n} := A \langle X \rangle_T := (A\langle X \rangle_T,A\langle X \rangle^+_T)
\end{gather*}
and a natural homomorphism of Huber pairs $(A,A^+) \rightarrow (A\langle X \rangle_T,A\langle X\rangle^+_T)$. Note that this construction is universal with respect to certain morphisms of Huber pairs  $(A,A^+)\rightarrow(B,B^+)$ where $(B,B^+)$ is complete.
} 
\end{defn}

\begin{exmp}\label{exmp. Huber pairs}
\emph{
\begin{enumerate}
\item $(k,k^o)$ is a Huber pair where we have $k^o= \lbrace c \in k \ | \ \val c \leq 1 \rbrace$.
\item $k\langle X_1,...,X_n \rangle$ is the \emph{Tate algebra over $k$} in $n$ variables. Let $c \in k$ be a topologically nilpotent unit. By choosing $(k^o \langle T_1,...,T_n \rangle, c \cdot k^o \langle T_1,...,T_n \rangle$ as a pair of definition $k\langle X_1,...,X_n \rangle$ can be considered to be a Huber ring.
\item For each Huber ring $A$ the subring of power-bounded elements is a ring of integral elements, in particular $( k \langle T_1,...,T_n \rangle, k^o \langle T_1,...,T_n \rangle)$ is a Huber pair.
\end{enumerate}
}
\end{exmp}

\begin{defn}\label{def. Huber pairs topologically of finite type /k}
\emph{
Let $f:(A,A^+)\rightarrow (B,B^+)$ be a morphism of Huber pairs.
\begin{enumerate}
\item
$f$ is said to be \emph{topologically of finite type} if there exists an $n \in \mathbb{N}$, $T_1,...,T_n$ finite subsets of $A^+$ such that $T_i \cdot A$ is open in $A$ for all $i$ and a morphism of Huber pairs $g:A\langle X_1,...,X_n \rangle_{T_1,...,T_n} \rightarrow B$ with $f= g \circ h$, such that $g$ is surjective, continuous, open and $B^+$ is the integral closure of $f(A^+)$ in $B$. Here $h$ denotes the canonical morphism of Huber pairs $A \rightarrow A\langle X_1,...,X_n \rangle_{T_1,...,T_n}$.
\item $f$ is called \emph{strictly topologically of finite type} if $f$ is topologically of finite type and one can choose $g$ in the definition above to have domain  $A\langle X_1,...,X_n \rangle$ for some $n \in \mathbb{N}$.
\item
 $(A,A^+)$ is said to be \emph{topologically of finite type over $k$} if there exists a homomorphism of Huber pairs $\pi:(k,k^o)\rightarrow (A,A^+)$ which is topologically of finite type.
\end{enumerate}
}
\end{defn}

\begin{rem}\label{rem. huber pairs of finite type equal k affinoid }
Let $f:(A,A^+)\rightarrow (B,B^+)$ be a morphism of Huber pairs.
\emph{\begin{enumerate}
\item If $A$ is a Tate ring and $B$ is complete, then $f$ is strictly topologically of finite type if and only if $f$ is topologically of finite type (cf. \cite{We1} Proposition 6.36).
\item As we will see later, this assertion implies that $(A,A^+)$ is of topologically finite type over $k$ if and only if there exists a surjective and continuous $k$-algebra homomorphism $k\langle X_1,...,X_n \rangle \rightarrow A$ and $A^+=A^o$ (cf. \ref{lem. unique topological of finite type}). This observation is important since it provides an equality between the fundamental algebraic objects in the theory of strictly $k$-analytic Berkovich spaces and the theory of adic spaces that are locally of finite type over $k$.
\end{enumerate}
}
\end{rem}

\subsection{The adic spectrum of a Huber pair}

In this section we associate a certain topological space $X=\Spa(A,A^+)$ to a given Huber pair $(A,A^+)$. $X$ will consist of equivalence classes of valuations of $A$ and the elements of $A$ will be considered as functions on $X$, which are bounded with respect to $A^+$. Beginning with the construction of such space, we will state important properties afterwards and give a short example at the end.

\begin{defn}\label{continuous valuations of a topological ring}
\emph{Let $A$ be a topological ring.
A \emph{valuation of $A$} is a map $\val \cdot :A \rightarrow \Gamma \cup
\lbrace 0 \rbrace$, where $\Gamma$ is a (multiplicatively written) totally ordered group such that we have 
\begin{enumerate}
\item $\val{f+g} \leq \max \lbrace \val f , \val g \rbrace$,
\item $\val{fg}=\val f \val g$ and
\item $\val 0 = 0$ and $\val 1 =1$
\end{enumerate}
for all $f,g \in A$.
If $x$ is a valuation of $A$ and $f\in A$, we sometimes write $\val{f(x)}$ instead of $x(f)$ to stress  the interpretation of $f$ as a function on the 'space of valuations of $A$'. This will be made more precise in the following.
}
\end{defn}

\begin{defn}\label{def. defs for valuations}
\emph{
Let $v: A \rightarrow \Gamma \cup \lbrace 0 \rbrace$ be a valuation of a topological ring $A$.
\begin{enumerate}
\item $v$ is said to be \emph{continuous} if for all $\gamma \in \Gamma$ the set $\lbrace f \in A \ | \ v(f) < \gamma \rbrace$ is open in $A$.
\item The \emph{support of $v$} is defined as $\supp(v):= \lbrace a \in A \ | \ v (a)=0 \rbrace$.
\item The \emph{value group of $v$} is the subgroup $\Gamma_v$ of $\Gamma$ generated by $v(A)\cap \Gamma$. If we consider a valuation $v$ as above and it is not stated differently, we will always assume  $\Gamma$ to be the valuation group of $v$.
\item For a second valuation $v': A\rightarrow \Gamma' \cup \lbrace 0 \rbrace $ we say that $v$ and $v'$ are \emph{equivalent} if for all $f,g\in A$ one has $v(f) \leq v(g) \Leftrightarrow v'(f)  \leq v'(g)$. In this case we write $v \sim v'$. Equivalently $v \sim v'$ if there exists an isomorphism of ordered monoids $f:\Gamma \cup \lbrace 0 \rbrace \rightarrow \Gamma' \cup \lbrace 0 \rbrace$ such that $f \circ v = v'$.\\  Note that equivalent valuations have the same support.
\item The \emph{rank} of $v$ is defined as the \emph{rank} of its value group $\Gamma_{v}$ which is the cardinality of convex subgroups of $\Gamma_{v}$ that are not equal to $\lbrace 1 \rbrace$. (Remember that for a totally ordered group $G$ a subgroup $H$ of $G$ is called \emph{convex} if for all $g \in G$ and $g \in H$ with  $h \leq g \leq 1$ we already have $g \in H$.)  Note that equivalent valuations are of the same rank.
\item The $v$-topology on $A$ (by abuse of notation we identify $A$ with its underlying ring without any topology) is the ring topology on $A$ such that a fundamental system of open neighbourhoods of  $0 \in A$ is given by $(\lbrace a \in A \ | \ v (a) < \gamma \rbrace)_{\gamma \in \Gamma}$ (cf. \cite{We1} Proposition 5.39). 
This means that $v$ is continuous if and only if $\id:A\rightarrow A'$ is continuous where $A'$ is $A$ endowed with the $v$-topology.
Moreover the obtained topology on $A$ does not change if we pass to an equivalent valuation.
\end{enumerate}
}
\end{defn}

\begin{rem}\label{rem. extension of valuations to field}
\emph{
Let $v :A \rightarrow \Gamma \cup \lbrace 0 \rbrace$ be a valuation of a topological ring $A$. In this case we have the field $\kappa(v) := \text{Frac}(A/\supp(v))$  since $\supp(v)$ is a prime ideal in $A$. There exists a unique extension as a valuation $\tilde{v}$ of $v$ to $\kappa(v)$ and the valuation groups coincide. In particular $\tilde{v}$ and $v$ are of the same rank. 
 We also define \emph{the valuation ring of $v$ in $\kappa(v)$} as $A_{v}:= \lbrace a \in \kappa(v) \ | \ \tilde{v}(a)\leq 1 \rbrace$. Note that $A_{v}$ is open with respect to the $\tilde{v}$-topology on $\kappa(v)$. Here the whole construction does not change if we pass to an equivalent valuation.
}
\end{rem}

\begin{defn}\label{adic spectrum of a Huber ring}
\emph{Let $(A,A^+)$ be a Huber pair. 
\begin{enumerate}
\item	$\Spa(A,A^+)$ denotes the \emph{adic spectrum of $(A,A^+)$} and is defined as 
\begin{gather*}
\Spa(A,A^+):= \Bigl\lbrace \vald \ \text{valuation on $A$} \ ; \begin{tabular}{lcr}
 $\vald \ \text{is continuous and} $\\
  $ \val f \leq 1 \ \text{for all} \ f \in A^+$\\
 \end{tabular} \Bigr\rbrace/ \sim.
\end{gather*}
\item We endow $\Spa(A,A^+)$ with the topology generated by subsets of the form 
\begin{gather*}
\lbrace x \in \Spa(A,A^+) \ | \ \val{f(x)} \leq \val{g(x)} \rbrace
\end{gather*}
where $f,g \in A$.
\end{enumerate}
}
\end{defn}

Now we briefly come back to the notion generizations which we introduced in the first section. 

\begin{defn}\label{defn generizations}
\emph{
Let $(A,A^+)$ be a Huber pair and $X=\Spa(A,A^+)$ the associated adic spectrum. For $x$ and $x' \in X$ such that $x' \in \overline{\lbrace x \rbrace} $ and $\supp(x)=\supp(x')$ we say that $x$ is a \emph{vertical generization} of $x'$. 
}
\end{defn}

\begin{rem}\label{rem. vertical genrization}
\emph{
Let $(A,A^+)$ be a Huber pair and $X=\Spa(A,A^+)$ the associated adic spectrum.
\begin{enumerate}
\item In general not all generizations that appear in $X$ are vertical. However, one can show that this is the case if $X$ is \emph{analytic} (cf. \ref{def. analytic adic spaces}). Since all adic spaces associated to Berkovich spaces have this property, vertical generizations play a central role in our discussion.
\item If $x,x' \in X$ such that $x'$ is non-trivial and a vertical generization of $x$, then $x$ and $x'$ induce the same topology on $\kappa(x)=\kappa(x')$ (cf. \cite{We1} Proposition 5.45).
\item Let $x\in X$. Note that the vertical generizations of $x$ are in bijection with the convex subgroups of the value group of $x$ (cf. \cite{We1} Remark 4.12) which are totally ordered by inclusion. Therefore if $x$ only allows vertical generizations in $X$, the set of all generizations of $x$ is totally ordered. 
\end{enumerate}
}
\end{rem}

The last observation shows that adic spectra of certain Huber pairs are \emph{valuative spaces} which will be introduced in §5.

The following theorem contains some of the most important properties for adic spectra. In particular, it establishes the importance of so called \emph{rational subsets} which play a central role when we endow the adic spectrum of a Huber pair with a structure sheaf.

\begin{theo}\label{theo. adic spectrum is spectral space}
Let $(A,A^+)$ be a Huber pair. Then $X:= \Spa(A,A^+)$ is a spectral (and in particular coherent) space. Moreover, the sets
\begin{gather*}		
X\biggl(\frac{T}{s}\biggr)	:=\lbrace  x \in \Spa(A,A^+) \ | \ \val{t(x)} \leq \val{s(x)} \neq 0 \ \text{\emph{for all}} \ t\in T\rbrace,
\end{gather*}
where $s\in A $ and $T\subseteq A$ is a finite subset such that  $T\cdot A$ is an open ideal of $A$, form a basis of quasi-compact subsets of $X$ that is stable under finite intersections.
\end{theo}
\begin{proof}
\cite{Hu2} Theorem 3.5.
\end{proof}

Subsets of the particular form in the theorem above are called $\emph{rational}$.

\begin{rem}\label{rem. morphisms of affinoid adic spaces}
\emph{Each morphism of Huber pairs $\varphi :(A,A^+)\rightarrow (B,B^+)$ induces a continuous map $\Spa(\varphi):\Spa(B,B^+)\rightarrow \Spa(A,A^+)$ by $x \mapsto x \circ \varphi$. 
}
\end{rem}

One can show that there is a homeomorphism  between the adic spectrum of a Huber pair $(A,A^+)$ and the adic spectrum of its completion $(\hat{A},\hat{A}^+)$. This map induces a correspondence of rational subsets of the respective spectra (cf. \cite{Hu2} Proposition 3.9). Later we will endow an adic spectrum $\Spa(A,A^+)$ with a pre-sheaf of complete topological rings which is initially defined on the collection of rational subsets of $\Spa(A,A^+)$. Hence we will also get an isomorphism in the right category of ringed spaces between $\Spa(A,A^+)$ and $\Spa(\hat{A},\hat{A}^+)$. So in the following we can and will assume any Huber pair to be complete.

The following easy example of an adic spectrum will play an important role when we consider adic spectra of $k$-algebras:

\begin{exmp}\label{exmp. Spa(k,k^o)}
\emph{$\Spa(k,k^o)$ consists of a single point, namely the one corresponding to the fixed non-archimedean rank $1$ valuation $\vald $ on $k$. We will provide the arguments for this claim within the next lines. It is clear that the equivalence class of $\vald$ is an element of $\Spa(k,k^o)$. Now assume that there is a continuous valuation $v: k \rightarrow \Gamma \cup \lbrace 0 \rbrace  $ such that $k^o \subseteq A_v:= \lbrace a \in k \ | \ v(a) \leq 1 \rbrace$ (so $A_v$ denotes the valuation ring associated to $v$). In this case $A_v$ is a $k^o$-subalgebra of $k$. If $k^o=A_v$, then $v$ is equivalent to $\vald$. So assume that there is $c \in A_v$ such that $\val c > 1$. Let $a \in k$ be arbitrary. Then there is an $n\in \mathbb{N}$ with $\val{ a c^{-n}}\leq 1$ and hence $a= a c^{-n} c^n \in k^o[c] \subseteq A_v$. So we have $A_v = k$ and therefore $v$ is trivial. Since $v$ is continuous, $\lbrace 0 \rbrace  = \lbrace b \in k \ | \ v(b)<1 \rbrace$ is open in $k$. But this is false and hence we get a contradiction to $k^o \subsetneq A_v$.\\
Therefore for any Huber pair $(A,A^+)$ we have a map $\Spa(A,A^+) \rightarrow \Spa(k,k^o)$ which is continuous. It will have a special meaning if $A$ is a $k$-algebra, since (as we will explain later in the discussion) in this case it corresponds to the map $k \rightarrow A$.
}
\end{exmp}

\subsection{Localization constructions in \textit{$\Spa(A,A^+)$}}

Our next goal is to provide the adic spectrum $X$ of a Huber pair $(A,A^+)$ with a particular structure sheaf of complete topological rings. The strategy is to find certain Huber pairs such that their spectra can be identified with rational subsets of $X$. Similarly to the theory of affine prime schemes, this will be achieved by a certain localization construction which is introduced in this subsection.

\begin{const}(Localization)\label{con. localization in a Huber ring}
\emph{
Let $A$ be a Huber ring with pair of definition $(A_0,I)$. Choose $s \in A$ and $\emptyset \neq T= \lbrace t_1,..,t_n \rbrace \subseteq A$ such that the ideal in $A$ generated by $T$ is open. 
Set $D:=A_0[\frac{t_1}{s},...,\frac{t_n}{s}]\subseteq A_s = A[s^{-1}]$ and endow $A_s$ with the ring topology such that $(I^n \cdot D)_{n\geq 1}$ is a fundamental system of open neighbourhoods of $0$. Let $A_s$ endowed with this particular topology be denoted by $A(\frac{T}{s})$. Note that this is a Huber ring with pair of definition $(D,I\cdot D)$. Its completion, which again is a Huber ring, is denoted by $A\langle\frac{T}{s} \rangle$. The ring homomorphism $\rho: A\rightarrow A\langle\frac{T}{s} \rangle $ is continuous and has the following universal property: \\
Let $\varphi : A \rightarrow B$ be a morphism of complete Huber rings, such that $\varphi(s) \in B^\times$ and $\frac{\varphi(t)}{\varphi(s)}\in B^o$ for all $t \in T$. Then there exists a unique morphism of complete Huber rings $\pi : A\langle\frac{T}{s} \rangle \rightarrow B$ such that $\pi \circ \rho = \varphi$. Note that $A\langle\frac{T}{s} \rangle \cong A \langle X \rangle _T /\overline{(1-sX)}$ since both rings have the same universal property.\\
Let $(A,A^+)$ be a Huber pair and let $A(\frac{T}{s})^+$ denote the integral closure of $A^+[\frac{t_1}{s},...,\frac{t_n}{s}]$ in $A(\frac{T}{s})$. Then $(A(\frac{T}{s}),A(\frac{T}{s})^+)$ is a Huber pair as well as its completion $A\langle\frac{T}{s} \rangle := (A\langle\frac{T}{s} \rangle,A\langle\frac{T}{s} \rangle^+)$. The latter is universal for morphisms of complete Huber pairs $\varphi :(A,A^+) \rightarrow (B,B^+)$ such that $\varphi(s) \in B^\times$ and $\frac{\varphi(t)}{\varphi(s)}\in B^+$ for all $t \in T$. \\
Note that the obtained morphism of Huber pairs $A \rightarrow A\langle \frac{T}{s}\rangle$ is topologically of finite type.
}
\end{const}

The following lemma provides some important consequences of the construction above. In particular, it shows that rational subsets of adic spectra can be described by adic spectra of Huber pairs which we constructed above.

\begin{lem}\label{lem. embeddings of rational subsets}
Let $\varphi:= \Spa(\rho):\Spa(A\langle\frac{T}{s}\rangle,A\langle\frac{T}{s}\rangle^+) \rightarrow \Spa(A,A^+)=X$ be as above.
\begin{enumerate}
\item $\varphi$ is an open topological immersion with image $U:=X(\frac{T}{s})$ ,
\item $\varphi$ is universal for continuous ring homomorphisms $\varphi: A \rightarrow B$ (where $B$ is a complete Huber ring) such that $\Spa(\varphi)$ factor through $U$, 
\item $\varphi$ induces a bijection between the rational subsets of $\Spa(A\langle\frac{T}{s}\rangle,A\langle\frac{T}{s}\rangle^+)$ and the rational subsets of $X$ that are contained in $X(\frac{T}{s})$.
\end{enumerate}
\end{lem}
\begin{proof}
\cite{Hu3} Proposition 1.3 and Lemma 1.5.
\end{proof}

\subsection{Affine adic spaces}
In this section we want to make use of our previous observations to define a pre-sheaf on $\Spa(A,A^+)$. It is important to note that this pre-sheaf will not always be a sheaf. However, we also state a theorem that ensures this property for all Huber pairs that are important for our later comparison to Berkovich spaces. Note that as explained before, we assume all Huber pairs to be complete. By abuse of notation, we sometimes do not distinguish between an element of $\Spa(A,A^+)$ and a representative of the equivalence class.

\begin{const}\label{construction: structure sheaf on Spa(A,A^+)}
\emph{
Let $(A,A^+)$ be a Huber pair and let $X=\Spa(A,A^+)$ be the associated spectrum. For a rational subset $U=X(\frac{T}{s})\subseteq \Spa(A,A^+)$ we set $\mt O _X (U):=A\langle \frac{T}{s}\rangle$ and $\mt O_X^+(U) := A\langle \frac{T}{s} \rangle^+$.
Using the universal property of $A\langle \frac{T}{s}\rangle$ one shows that the assignments $U\mapsto \mt O_X(U)$ and $U\mapsto \mt O^+_X(U)$ provide well defined pre-sheaves of complete topological rings on the basis of rational subsets of $X$, where 
\begin{gather*}		
\mt O_X^+(U)= \lbrace f \in \mt O_X(U) \ | \ \val{f(x)}\leq 1 \ \text{for all} \ x\in U \rbrace \ \ (\ast)
\end{gather*} holds (cf. \cite{Hu3} Proposition 1.6 (iv)).
If $\mt O_X$ is a sheaf (in this case we will call $(A,A^+)$ a \emph{sheafy Huber pair}) on the basis consisting of rational subsets of $X$, then it can uniquely be extended to a sheaf of complete topological rings on the topological space $\Spa(A,A^+)$ using the usual limes construction. In this case, due to $(\ast)$, $\mt O_X^+$ is also a sheaf of complete topological rings on $\Spa(A,A^+)$. \\ Let $x\in X$. As usual the \emph{stalk of $x$ in $X$} is defined by
\begin{gather*}		
\mt O_{X,x}:=\underset{U \ni x \  \text{open}}{\text{colim}} \mt O_X(U) = \underset{U \ni x \  \text{rational}}{\text{colim}} \mt O_X(U).
\end{gather*}
Note that the colimes is taken in the category of rings and hence there is a priori no topology on
$\stalk$. Let $U$ be a rational subset of $X$ with $x \in U$. Then by \ref{lem. embeddings of rational subsets} we can uniquely extend $x$ to a valuation of $\mt O_X(U)$ and by the universal property of $\stalk$ we get a valuation $v_x$ of $ \stalk$. One shows that $\stalk$ is a local ring with maximal ideal $\supp(v_x)$.
}
\end{const}

One might ask if every Huber pair is sheafy and if not what are important examples of sheafy Huber pairs. The first question has to be denied (cf. \cite{BV} §4) but the following theorem provides a list of sufficient properties for a Huber pair to be sheafy.

\begin{theo}\label{theo. when is a huber pair sheafy? sufficient properties}
Let $(A,A^+)$ be a Huber pair. If $(A,A^+)$ satisfies at least one of the following properties, it is sheafy:
\begin{enumerate}
\item $A$ has a noetherian ring of definition;
\item $A$ is a Tate ring and $A \langle X_1 ,...,X_n \rangle$ is noetherian for all $n\in \mathbb{N}$;
\item $A$ has the discrete topology.
\end{enumerate}
\end{theo}
\begin{proof}
\cite{Hu3} Theorem 2.2 and \cite{We1} Theorem 8.27.
\end{proof}

\begin{rem}\label{rem: sheafy Huber pairs}
\emph{
A Tate ring which satisfies condition ii) of the theorem above is called \emph{strongly noetherian}. One can show that every Tate ring that is topologically of finite type over a strongly noetherian Tate ring is strongly noetherian as well (cf. \cite{We1} Remark 6.39). Since $k$ is a strongly noetherian Tate ring (cf. \cite{BGR} 5.2.6 Theorem 1), any Huber ring that is topologically of finite type over $k$ is strongly noetherian. Consequently, any Huber pair that is topologically of finite type over $k$ is sheafy by the theorem above.
}
\end{rem}

\subsection{Adic spaces}

In this section we globalize our previous construction of adic spectra to certain locally ringed spaces. Those spaces are abstract triples which locally look like adic spectra of sheafy Huber pairs.

\begin{defn}\label{def. category V}
\emph{The category $\mt V$ is defined as follows:\\
 An object of $\mt V$ is a  triple $(X, \mt O_X , (v_x)_{x\in X})$ consisting of the following data:
\begin{enumerate}
\item A topological space $X$,
\item a sheaf of complete topological rings on $X$ such that the stalk $\stalk$ is a local ring for all $x \in X$,
\item for each $x \in X$ an equivalence class of valuations on $\stalk$ denoted by $v_x$ such that the support of $v_x$ is the maximal ideal of $\stalk$.
\end{enumerate}
A morphism $(X, \mt O_X , (v_x)_{x\in X})\rightarrow (Y, \mt O_Y , (v_y)_{y\in Y})$ in $\mathcal{V}$ is a pair $(f,f^b)$ consisting of 
\begin{enumerate}
\item a continuous map $f: X\rightarrow Y$,
\item a morphism of sheaves of complete topological rings $f^b: \mt O_Y \rightarrow f^*\mt O_X$ such that $v_{f(x)} = v_x \circ f_x^b$ for all $x \in X$.
\end{enumerate}
}
\end{defn}

\begin{rem}\label{rem. sheafy huber pairs belong to V}
\emph{
\begin{enumerate}
\item  Note that the definition of morphisms in $\mt V$ implies that the induced maps on the stalks $f_x^b$ are local morphisms of local rings. 
\item Let $(A,A^+)$ be a sheafy Huber pair. Then 
\begin{gather*}		
X:=\Spa(A,A^+):=(X,\mt O_X, (v_x)_{x \in X})
\end{gather*}
defined as in (\ref{construction: structure sheaf on Spa(A,A^+)}) is an object in $\mt V$.
\end{enumerate}
}
\end{rem}

\begin{defn}\label{def. adic space}
\emph{An \emph{adic space} is an object of $\mt V$ that is locally isomorphic to $\Spa(A,A^+)$ for sheafy Huber pairs $(A,A^+)$. The category of adic spaces is the full subcategory of $\mt V$ whose objects are adic spaces. In turn the full subcategory of the category of adic spaces, whose objects are isomorphic to $\Spa(A,A^+)$ for some sheafy Huber pair $(A,A^+)$, is by definition the \emph{category of affinoid adic spaces}.
}
\end{defn}

\begin{rem}\label{rem: functor sheafy huber pairs to affinoid adic spaces}
\emph{
We obtain a functor $\Spa$ from the category of sheafy Huber pairs to the category of affinoid adic spaces by $\varphi: (A,A^+)\mapsto \Spa(A,A^+)$. Which sends a morphism of Huber pairs $(A,A^+) \rightarrow (B,B^+)$ to the pair $(f,f^b)$ where $f$ is induced by composition with $\varphi$ and $f^b$ is obtained the universal property of Huber pairs associated to rational subsets.
}
\end{rem}

The following theorem is important for our following discussion since it shows that all information of an affinoid adic space is already contained in its associated Huber pair and vice versa.

\begin{theo}\label{theo. fully faithful functor}
Let $X$ be an adic space and $(A,A^+)$ a complete sheafy Huber pair whose associated affinoid adic space $\Spa(A,A^+)$ is denoted by $Y$. Then we have a bijection
\begin{gather*}		
\Hom_{\text{adic spaces}}(X,Y) \rightarrow \Hom((A,A^+),(\mt O_X(X), \mt O_X(X)^+))\\
(f,f^b)\mapsto f^b_Y.
\end{gather*}
The morphisms on the right hand side are continuous ring homomorphisms $\varphi:A\rightarrow \mt O_X(X)$ such that $\varphi(A^+) \subseteq \mt O_X(X)^+$. 
In particular the functor $(A,A^+)\mapsto \Spa(A,A^+)$ from the category of complete sheafy Huber pairs to the category of affinoid adic spaces is an equivalence of categories.
\end{theo}
\begin{proof}
\cite{We1} Proposition 8.25.
\end{proof}

\subsection{Adic spaces locally of finite type over $k$}

This subsection provides a notion of adic spaces which satisfy certain finiteness conditions over $k$. In particular, all adic spaces we want to associate to certain Berkovich spaces will have this property.

\begin{defn}\label{def. morphisms locally of finite type}
\emph{\begin{enumerate}
\item
Let $f:X\rightarrow Y$ be a morphism of adic spaces. $f$ is called \emph{locally of finite type} if for every $x \in X$ there is an open affinoid neighbourhood $U=\Spa(B,B^+)$ of $x$ and an open affinoid subspace $V=\Spa(A,A^+)$ of $Y$ with $f(U) \subseteq V$ such that the induced morphism of Huber pairs $(A,A^+)\rightarrow (B,B^+)$ is topologically of finite type.
\item
An adic space $X$ is said to be \emph{locally of finite type over $k$} if the trivially given continuous map $X\rightarrow \Spa(k,k^o)$ is a morphism of adic spaces that is locally of finite type over $k$.
\end{enumerate}
}
\end{defn}

\begin{prop}\label{prop. properties of morphisms that are locally of finite type}
Let $\varphi : A\rightarrow B$ and $\psi :B \rightarrow C$ be morphisms of Huber pairs, then the following assertions hold:
\begin{enumerate}
\item If $\varphi$ and $\psi$ are topologically of finite type, then its composition $\psi \circ \varphi$ is locally of finite type.
\item If $\psi \circ \varphi$ is topologically of finite type, then $\psi$ is locally of finite type.
\end{enumerate}
\end{prop}
\begin{proof}
\cite{Hu3} Lemma 3.5. iv).
\end{proof}

\begin{prop}\label{prop: locally of finite type first prop}
Let $f:X \rightarrow Y$ be a morphism of adic spaces that is locally of finite type. Let $U \subseteq X$ and $V \subseteq Y$ be open subsets such that $f(U) \subseteq V$. Then the morphism of adic spaces $U \rightarrow V$ obtained by the restriction of $f$ is locally of finite type.
\end{prop}
\begin{proof}
Let $x \in U$ be arbitrary. By hypothesis we have an open affinoid neighbourhood $U'=\Spa( B',B'^+)$ of $x$ in $X$ and an affinoid open subset $V'=\Spa( A',A'^+)$ of $Y$ such that $f(U')\subseteq V'$ and the induced morphism of Huber pairs $g:(A',A'^+)\rightarrow (B',B'^+)$ is topologically of finite type. We choose a rational subset $R(\frac{T}{s})$ of $Y$ such that  $f(x) \in R(\frac{T}{s})\subseteq V \cap V'$ and a rational subset $R(\frac{M}{r})$ of $X$ such that $x \in R(\frac{M}{r}) \subseteq f^{-1}(R(\frac{T}{s}) \cap U' \cap U)$. So we have $f(R(\frac{M}{r}))\subseteq R(\frac{T}{s})$ and the rational sets are rational subsets of $(A',A'^+)$ and $(B',B'^+)$ respectively. Hence we have a commutative diagram
\begin{gather*}		
 \begin{xy}
  \xymatrix{
      A' \ar[r]^g \ar[d]^i    &   B' \ar[d]^j  \\
      A' \langle \frac{T}{s}\rangle \ar[r]^l             &   B' \langle \frac{M}{r} \rangle   
  }
\end{xy}
 \end{gather*} 
where $i$ and $j$ are topologically of finite type and so is $g$. Hence $j \circ g$ is topologically of finite type by \ref{prop. properties of morphisms that are locally of finite type} i) and hence $l$ is topologically of finite type by \ref{prop. properties of morphisms that are locally of finite type} ii). This shows the claim.

\end{proof}

\begin{prop}\label{prop: locally of finite type second prop}
Let $f:X \rightarrow Y$ be a morphism of adic spaces that is locally of finite type. Let $U= \Spa(B,B^+) \subseteq X$ and $V=\Spa(A,A^+)\subseteq Y$ be open affinoid subspaces with $f(U)\subseteq V$. Then the induced homomorphism of Huber pairs $(A,A^+)\rightarrow (B,B^+)$ is topologically of finite type.
\end{prop}
\begin{proof}
By \ref{prop: locally of finite type first prop} we can assume that $X$ and $Y$ are affinoid. But in this case the result is known from \cite{Hu1} Proposition 3.8.15.
\end{proof}

\begin{rem}\label{rem. locally of finite type over k morphisms}
\emph{
Let $X$ be an adic space that is locally of finite type over $k$ and let $U=\Spa(A,A^+)$ be an open affinoid subset. By \ref{prop: locally of finite type second prop} the induced morphism of Huber pairs $(k,k^o)\rightarrow (A,A^+)$ is topologically of finite type. In particular if $X$ is affinoid, then $X=\Spa(A,A^+)$ for a Huber pair $(A,A^+)$ that is topologically of finite type over $k$.	
}
\end{rem}

\subsection{Analytic adic spaces}
As we will see in the next sections, it is necessary that there are topologically nilpotent units in a ring to be able to do analytic calculation. Therefore the definition of analytic adic spaces is straightforward.

 \begin{defn}\label{def. analytic adic spaces}
 \emph{Let $X$ be an adic space. 
  \begin{enumerate}
\item An element $x \in X$ is said to be \emph{analytic} if there exists an open neighbourhood $x \in U$ such that $\mt O_X(U)$ contains a topologically nilpotent unit.
 \item 
 $X$ is said to be \emph{analytic} if every point of $X$ is analytic.
\end{enumerate} 
 }
 \end{defn}

\begin{exmp}\label{exmp. analytic adic space}
\emph{
\begin{enumerate}
\item
Let $A:=k \langle T_1,...,T_n \rangle$ be a Tate algebra. Then $\Spa(A,A^o)$ is an analytic adic space, since $A$ contains a topologically nilpotent unit.
\item More generally, any adic space, that is locally of finite type over $k$, is analytic. 
\end{enumerate}
}
\end{exmp}

The following properties of analytic adic spaces assure that we can consider such spaces a valuative spaces which are introduced in §5.

\begin{prop}\label{prop: properties of analytic adic spaces}
Let $f:X \rightarrow Y$ be a morphism between analytic adic spaces.
\begin{enumerate}
\item For $x \in X $ all generizations of $x$ in $X$ are vertical and in particular the set $G_x$ of generizations of $x$ is totally ordered;
\item A point $x \in X$ is a maximal point if and only if the valuation $v_x$ has rank $1$;
\item If $x \in X$ is a maximal point, then $f(x) $ is a maximal point of $Y$.
\end{enumerate}
\end{prop}
\begin{proof}
\cite{Hu4} Lemma 1.1.10.
\end{proof}

\section{Berkovich spaces}
We also recall the basic definitions in the theory of $k$-analytic Berkovich spaces. Note that the general Berkovich theory also permits such cases in which the fixed valuation on $k$ is trivial. Since we want to compare Berkovich spaces with adic spaces, we restrict ourselves to the non-trivially valued case. We omit most of the proofs and follow papers and summaries by Berkovich \cite{Be1}, \cite{Be2}, \cite{Be3}, Temkin \cite{Te1} and Conrad \cite{Co1}.

\subsection{The Berkovich spectrum of a Banach ring}

\begin{defn}\label{def. kommut banach norm/ ring k banach algebra}
\emph{Let $A$ be a ring.
\begin{enumerate}
\item A \emph{seminorm on $A$} is a map $\norm \cdot :A \rightarrow  \mathbb{R}^{\geq 0}$ such that the following properties are satisfied for all $f,g \in A$:
\begin{enumerate}
\item $\norm{0} =0$;
\item $\norm{fg}\leq \norm f \norm g$;
\item $\norm{f+g} \leq \norm f + \norm g$.
\end{enumerate}
\item A seminorm $\norm \cdot$ on $A$ is said to be a \emph{Banach norm} if $\norm f \neq 0$ whenever $f \neq 0$.
\item A seminorm $\norm \cdot$ on $A$ is called \emph{multiplicative} if $\norm {fg}=\norm f \norm g$ for all $f,g \in A$.
\item A \emph{Banach ring} is a ring $A$ endowed with a Banach norm $\norm \cdot$ such that the induced topology is complete.
\item A \emph{$k$-Banach algebra} is a Banach ring endowed with the structure of a $k$-algebra such that $\norm {af}=\val a \norm f $ for all $a \in k$ and $f \in A$.
\end{enumerate}
}
\end{defn}

\begin{rem}\label{rem normal triangle inequality implies ultrametric}
\emph{Note that for every multiplicative seminorm $\norm \cdot$ on a $k$-algebra $A$ which is compatible with the valuation on $k$, the usual triangle inequality implies the non-archimedean triangle inequality. Indeed, for $f,g \in A$ with $\norm f \leq \norm g$ and $n \in \mathbb{N}$ we have:
\begin{gather*}
\norm{f+g}^n=\norm{(f+g)^n} \leq \sum_{j=0}^n \val{\binom{n}{k}} \norm f^k \norm {g}^{n-k}
\leq (n+1) \norm g ^n.
\end{gather*}
The result follows by taking the n-th root of this inequality and let $n$ go to infinity.
}
\end{rem}



\begin{defn}\label{def. berkovich spectrum}
\emph{Let $(A,\normd)$ be a Banach ring. The \emph{Berkovich spectrum} of $A$, denoted by $\mathcal{M}(A)$, is defined as the set of all bounded, multiplicative seminorms on $A$, i.e. the set of all multiplicative seminorms $\val \cdot_x :A \rightarrow \mathbb{R}^{\geq 0}$ such that there exists a bounding constant $C>0$ with $\val f_x\leq  C \norm f$ for all $f \in A$. We endow $\mathcal{M}(A)$ with the weakest topology such that the evaluation maps $\ev_f: \mt M(A) \rightarrow \mathbb{R}^{\geq 0}, \ \vald _x \mapsto \val f_x $ are continuous for all $f \in A$.}
\end{defn}

\begin{theo}\label{theorem: berkovich spectrum is nonempty, compact}
For a non-trivial Banach ring $A$, the topological space $\mathcal{M}(A)$ is non-empty, compact and Hausdorff.
\end{theo}
\begin{proof}
\cite{Be2} Theorem 1.2.1
\end{proof}

\begin{rem}\label{rem morphisms of spectra}
\emph{Let $(A, \normd)$ be a Banach ring.
\begin{enumerate} 
\item Note that if we endow $A$ with an equivalent norm, the space $\mt M(A)$ does not change since equivalent norms are respectively bounded by constants. 
\item We want to consider the elements of $A$ as functions on $\mathcal{M}(A)$ hence for $\val \cdot _x \in \mathcal{M}(A)$ and $f \in A$ we write $\val{f(x)}$ for $\val f _x$.
Note that any bounded homomorphism of Banach rings $\varphi:B\rightarrow A$ induces a continuous map between the respective Berkovich spectra $\mathcal{M}(\varphi): \mathcal{M}(A)\rightarrow \mathcal{M}(B)$ via $\val \cdot_x\mapsto \val \cdot_x \circ \varphi$.
\item Let $\vald _x$ be in $\mt M(A)$ then $\val f _x \leq \norm f$ for all $f \in A$, i.e. the bounding constant in the definition above can always be chosen to be $1$. Indeed, assume that $1$ is not a bounding constant of $\vald _x$ with respect to $\normd$. Then we can find an $f \in A$  and $d > 1 $ with $\val f _x =d \norm f$. But then $\vald _x$ is not bounded at all, since for an arbitrary $C>0$, there exists an $n\in \mathbb{N}$ with $d^n>C$ and hence $\val {f^n}_x = \val f^n_x=d^n \norm f^n>C \norm{f^n}$.
\item Note that if $A$ is a $k$-Banach algebra with norm $\normd$ such that $\norm 1=1$ then all elements of $\mt M(A)$ extend the given valuation $\vald$ of $k$. Indeed: For $\val{\cdot}_x \in \mt M(A)$ and $0 \neq a \in k$ we have $\val{ a \cdot 1 }_x \leq \norm{a \cdot 1} = \val a $. Multiplying this inequality by $\val{ a} ^{-1} \cdot \val{a^{-1} \cdot 1}_x$ we get $\val{a^{-1}} \leq \val{a^{-1}\cdot 1}_x $. But the first inequality also holds for $a^{-1}$ and hence $\val{a^{-1}}=\val{a^{-1}}_x$, or equivalently $\val{a}=\val{a}_x$.
\end{enumerate}
}
\end{rem}

The following lemma provides a different description of the topology of the Berkovich spectrum of a $k$-Banach algebra. Although in the Berkovich theory it is more important to work with non-strict inequalities defining rational domains (cf. \ref{example rational domain}), this result is important to compare the topologies of affinoid adic and affinoid Berkovich spaces (cf. \ref{prop. map: spa a to m(a)}).

\begin{lem}\label{lem. topology of berkovich spectrum}
Let $A$ be a $k$-Banach algebra. A subbasis of the topology of $\mt M(A)$ is given by sets of the form $U_{f,g} := \lbrace x \in \mt M(A) \ | \ \val{f(x)}<\val{g(x)} \rbrace$ where $f,g \in A$.
\end{lem}
\begin{proof}
At first note that $U_{f,g}$ is open for all $f,g \in A$, since we have 
\begin{gather*}		
U_{f,g}=\bigcup_{\alpha \in \mathbb{R}} (ev_f^{-1}([0,\alpha))\cap (\mt M(A) \setminus ev_g^{-1}([0,\alpha ])).
\end{gather*}
Now let $a<b \in \mathbb{R}^{\geq 0}$, $f \in A$ be arbitrary and $V:=ev_f^{-1}((a,b))$. Let $z \in V$, i.e. $a<\val{f(x)}<b$. Since $\sqrt{\val{k^{\times}}}$ is dense in $\mathbb{R}^{\geq 0}$ we can find $n,m \in \mathbb{N}$ and $c_1,c_2 \in k$ such that $a<\sqrt[n]{\val{c_1}}<\val{f(z)}<\sqrt[m]{\val{c_2}}<b$. Then $W:=U_{c_1^m,f^{n+m}}\cap U_{f^{n+m},c_2^n}$ is an open neighbourhood of $z$ which is contained in $V$.
\end{proof}

\subsection{k-affinoid algebras and k-affinoid spaces}

In this subsection, we introduce the fundamental algebraic objects for the theory of Berkovich spaces.

\begin{defn}\label{def. def. generalized tate algebra}
\emph{Let $A$ be a $k$-Banach algebra with Banach norm $\norm \cdot_A$. For $r_1,\ldots, r_n>0$ the set
\begin{gather*}
A\lbrace r_1^{-1}T_1,...,r_n^{-1}T_n \rbrace := \\ \Bigl\lbrace f= \sum_{\nu\in\mathbb{N}^n_0} a_{\nu} T^{\nu} \ | \ a_{\nu} \in A \ \text{and} \ \norm{a_{\nu}}_A \cdot r^{\nu} \rightarrow 0 \ \text{as} \ |\nu| \rightarrow \infty \Bigr\rbrace
\end{gather*}
($\val \nu$ denotes the multi-index norm of $\nu$, i.e. $\val \nu = \nu_1+...+\nu_n$) is a $k$-Banach algebra with respect to the Banach norm 
\begin{gather*}
\norm{\sum_{\nu\in\mathbb{N}^n_0} a_{\nu} T^{\nu}}:= \max_{\nu \in \mathbb{N}^n_0} \norm{a_{\nu}}_A \cdot r^{\nu}.
\end{gather*} 
For brevity it is sometimes denoted by $A\lbrace r^{-1}T\rbrace$.
If we choose $(A, \normd_A)$ to be $(k, \vald)$, it is called \emph{generalized Tate algebra}.
}
\end{defn}

\begin{rem}\label{rem.first comarison of tate algebra on both sides}
\emph{\begin{enumerate}
\item  The Berkovich spectrum associated to a generalized Tate algebra is considered as the $n$-dimensional Berkovich polydisc with radius $(r_1,...,r_n)$ and therefore those polydiscs will be the fundamental building blocks of $k$-analytic Berkovich spaces which we want to define in this section.
\item
In \ref{defn. restricted power series} we have defined a different generalization of a Tate algebra. Note that if we choose $(A,\normd)=(k,\vald)$ in the definition above and $(r_1,...,r_n)=(1,...,1)$, we have $k\lbrace T_1,...,T_n \rbrace = k \langle T_1,...,T_n \rangle$. We will use both notions respectively to stress if we are situated in the Berkovich or in the adic setting.
\end{enumerate}
}
\end{rem}

Analogously to Banach spaces over $\mathbb{C}$ the following lemma states that a linear map between $k$-Banach spaces is bounded if and only if it is continuous. For the convenience of the reader we recall the proof which is nearly verbatim to the classic case.

\begin{lem}\label{lem. bounded means continuous}
Let $\varphi :(A,\normd) \rightarrow (B, \normd')$ be a linear map of $k$-Banach spaces (recall that a $k$-Banach space is a $k$-vector space, endowed with a $k$-vector space norm; in particular, any $k$-Banach algebra is a $k$-Banach space). Then $\varphi$ is continuous if and only if it is bounded. 
\end{lem}
\begin{proof}
It is clear that $\varphi$ is continuous if it is bounded. So assume that $\varphi$ is not bounded. Let $c \in k$ with $\val c < 1$  and $N\geq 1$ a natural number such that $\frac{1}{N+1}<\val c \leq \frac{1}{N}$. Since $\varphi$ is not bounded, there exists a sequence $(x_n)_{n \in \mathbb{N}}$ in $A$ such that $\norm{ \varphi (x_n)}'> \val c ^{-n} \norm{x_n}$ for all $n \in  \mathbb{N}$. Let $0< \epsilon < \frac{1}{N+1}$ and let $U':= \lbrace x \in V' \ | \ \norm{x}'< \epsilon \rbrace$. If we can show that there is no neighbourhood of $0$ in the preimage of $U'$ under $\varphi$ we know that $\varphi$ is not continuous and hence we are done. So let $\delta > 0$ be arbitrary and choose $n \in \mathbb{N}$ with $\val c ^{-n}> \frac{1}{\delta}$. Now choose $l \in \mathbb{N}$ minimal such that  $\norm{c^lx_n}< \delta$ (this implies $\frac{\delta}{N+1} \leq \norm{c^lx_n}$). We have
\begin{gather*}		
 \norm{\varphi(c^lx_n)}'=\val{c}^l\norm{\varphi(x_n)}'> \val c ^l\val c^{-n}  \norm{x_n}>\frac{1}{\delta}\delta \frac{1}{N+1} >\epsilon.
 \end{gather*} 
Therefore $c^lx_n$ is not an element of $\varphi^{-1}(U')$ and $\varphi$ is not continuous.
\end{proof}

The $k$-algebras defined in the following are the central algebraic objects in the theory of $k$-analytic Berkovich spaces.

\begin{defn}\label{def. k-affinoid algebra}
\emph{A  \emph{$k$-affinoid algebra} is a $k$-Banach algebra $(A, \normd)$ such that there is a surjective $k$-algebra homomorphism $\pi:k\lbrace r_1^{-1}T_1,...,r_n^{-1}T_n \rbrace \rightarrow  A$ for some generalized Tate algebra  $k\lbrace r_1^{-1}T_1,...,r_n^{-1}T_n \rbrace$.
$A$ is said to be \emph{strictly $k$-affinoid} if $(r_1,...,r_n)$ can be chosen to be $(1,...,1)$. 
}
\end{defn}

\begin{rem}\label{admissible homomorphismus}
\emph{
In the general Berkovich theory, where one permits a trivially valued $k$, one requires the $k$-algebra homomorphism $\pi$ in the definition above to be \emph{admissible}. This means that $\pi$ induces an isomorphism of Banach spaces $k\lbrace r_1^{-1}T_1,...,r_n^{-1}T_n\rbrace/\ker(\pi)  \rightarrow  A$, where $k\lbrace r_1^{-1}T_1,...,r_n^{-1}T_n \rbrace/\ker(\pi)$ is endowed with the quotient norm. We can omit this demand by the Open Mapping Theorem for such spaces, which ensures that any such morphism is admissible.
}
\end{rem}

We will later see that there is a one-to-one correspondence between strictly $k$-affinoid algebras and Huber pairs that are topologically of finite type over $k$. From this observation we will deduce that certain Berkovich spaces associated to strictly $k$-affinoid algebras can be associated to particular adic spaces that are locally of finite type over $k$.

\begin{rem}\label{rem. represetation of k affinoid algebras is not unique}
\emph{
For a $k$-affinoid algebra $A$ we do not explicitly name its norm since the Berkovich spectrum of $A$ is invariant under passing to an equivalent norm. By \emph{a norm on $A$} we mean a $k$-algebra norm that is in the equivalence class of the residue norm induced by a representation (which all are equivalent). For a norm on $A$ given by the residue norm $\normd$ of a representation $k\lbrace r_1^{-1}T_1,...,r_n^{-1}T_n \rbrace \rightarrow  A$, we have $\norm 1 =1$. Hence we know that all elements of $\mt M(A)$ extend the valuation on $k$ (cf. \ref{rem morphisms of spectra} iv)).
}
\end{rem}

\begin{rem}\label{rem. bounded brought together}
\emph{
Let $A$ be a $k$-affinoid algebra and $\norm \cdot$ a norm on $A$.
\begin{enumerate}
\item
A subset $B \subseteq A$ is bounded with respect to the usual definition of boundedness for normed spaces if and only if it is bounded with respect to definition \ref{def. Huber ring / Tate ring} 	iii). (For this remark differing from our fixed notation, we assume the product of two sets to be as in \ref{def. Huber ring / Tate ring} iii).) Indeed, if there is a real number $r > 0$ such that $\norm x < r$ for all $x \in B$, we have $B \cdot B_{\epsilon \frac{1}{r}} \subseteq B_\epsilon$ for all $\epsilon \in \mathbb{R}^{\geq 0}$ ($B_a$ denotes the open ball with center $0$ and radius $a$ with respect to $\norm \cdot$). Conversely, assume that $B$ is not bounded with respect to $\norm \cdot$. Consider an arbitrary $r > 0$. Let $\epsilon > 0$ be arbitrary as well. Choose $c \in k^{\times}$ such that $\val c < \epsilon$ and $x \in B$ with $\norm x> \frac{r}{\val c}$. Then we have $\norm{ c \cdot x }= \val c \norm x > r$. This means $B\cdot B_\epsilon \nsubseteq B_r$ for all $\epsilon > 0$ and hence $B$ is not bounded with respect to definition \ref{def. Huber ring / Tate ring} iii).
\item Let $\varphi : A\rightarrow B$ be a bounded map of $k$-affinoid algebras. Then $\varphi$ maps bounded sets to bounded sets and in particular we have $\varphi(A^o) \subseteq B^o$.
\end{enumerate}
}
\end{rem}
\begin{defn}\label{def. category k-affinoid spaces}
\emph{A \emph{$k$-affinoid space} $X$ is the Berkovich spectrum of some $k$-affinoid algebra $A$, i.e. $X=\mathcal{M}( A)$. A \emph{morphism of $k$-affinoid spaces} is a map $\mathcal{M}(B) \rightarrow \mathcal{M}(A)$ induced by a bounded homomorphism of $k$-Banach algebras $A\rightarrow B$ (cf. \ref{rem morphisms of spectra} i)). 
In other words, the \emph{category of $k$-affinoid spaces}, which we denote by $k$-$Aff$, is by definition the category dual to the category of $k$-affinoid algebras with bounded homomorphisms. We also define the category of \emph{strictly $k$-analytic spaces} to be the full sub category of  $k$-$Aff$ whose objects are the $k$-affinoid spaces corresponding to strictly $k$-affinoid algebras.
}
\end{defn}

$k$-affinoid spaces are the local objects in Berkovich's theory which are glued together to create general analytic Berkovich spaces.

\subsection{Affinoid domains}

In this subsection we focus on certain closed subsets of $k$-affinoid spaces which itself can be regarded as spectra of $k$-affinoid algebras.

In this section we fix a $k$-affinoid space $X=\mathcal{M}(A)$.

\begin{defn}\label{def. affinoid domain}
\emph{
\begin{enumerate}
\item
An \emph{affinoid domain in $X$} is a closed subset $V\subseteq X$ with a morphism of $k$-affinoid spaces $\varphi : \mathcal{M}(A_V)\rightarrow X$ with $\im(\varphi)\subseteq V$ such that any morphism of $k$-analytic spaces $\psi: Y\rightarrow X$ with image contained in $V$ factors uniquely through $\mathcal{M}(A_V)$.
\item An affinoid domain $V$ is said to be a \emph{strictly affinoid domain} if the $k$-affinoid algebra $A_V$ in the definition above can be chosen to be a strictly $k$-affinoid algebra.
\item
A morphism of $k$-affinoid spaces $\varphi:Y\rightarrow X$ is said to be an \emph{affinoid domain embedding} if it makes $\im (\varphi)$ to an affinoid domain in $X$.
\end{enumerate}
}
\end{defn}

\begin{theo}\label{theo. important properties of affinoid domain}
Let $X= \mt M(A)$ be a $k$-affinoid space and $V$ be an affinoid domain in $X$. The morphism of $k$-affinoid spaces $\varphi_V :\mt M (A_V)\rightarrow V \subseteq X$ is a homeomorphism.
\end{theo}
\begin{proof}
\cite{Be2} Proposition 2.2.4.
\end{proof}

\begin{rem}\label{rem. to addinoid domains}
\emph{\begin{enumerate}
\item
Let $V$ be an affinoid domain in $X$. The universal property of $\mt M(A_V)$ shows that $A_V$ is unique up to a unique bounded isomorphism of $k$-Banach algebras. Hence $A_V$ is well defined.
\item In the original paper of Berkovich (\cite{Be1}) and in the work of other authors (e.g. Fujiwara and Kato in \cite{FK}) affinoid domains are defined to satisfy a stronger universal property, also including morphisms from $k'$-affinoid spaces for a complete isometric extension $k\subseteq k'$. However, it can be shown that this is equivalent to the definition given in this paper (cf. \ref{prop. universal property holds for all k- banach algebras}). 
\end{enumerate}
}
\end{rem}

A key role for the whole theory of $k$-analytic Berkovich spaces plays the following notion of special affinoid domains:

\begin{exmp}\label{example rational domain}
\emph{\begin{enumerate}
\item
Let $g,f_1,...,f_n \in  A$ such that the ideal in $A$ generated by those elements is equal to $A$. Let $p=(p_1,...,p_n)$ be an $n$-tuple of positive numbers. The set 
\begin{gather*}
 V:=X(p^{-1} \frac{f}{g}):= \lbrace x \in X \ | \ \val{ f_i(x)}\leq p_i \val{g(x)} \ \text{for all} \ 1\leq i\leq n \rbrace
\end{gather*}
 is an affinoid domain in $X$ which is represented by the homomorphism 
\begin{gather*}		 
 A \rightarrow A_V:=A\lbrace p^{-1} \frac{f}{g} \rbrace := A\lbrace p_1^{-1} T_1,...,p_n^{-1} T_n \rbrace /(gT_i-f_i)_{i=1,...,n}.
 \end{gather*} 
  Affinoid domains of this type are called \emph{rational domains}. 
\item Let $B$ be a $k'$-affinoid algebra for some complete isometric extension $k\subseteq k'$ and $\varphi :A \rightarrow B$ be a bounded homomorphism of $k$-algebras. Let $V$ be a rational domain in $\mt M(A)$. Then $\varphi^{-1}(V)$ is a rational domain in $\mt M(B)$ which is defined by the same inequalities. For example $\varphi^{-1}(V)=\lbrace y \in \mt M(B) \ | \ y(\varphi(f_i))\leq p_i \cdot y(\varphi(g)) \ \text{for all} \ 1\leq i \leq n \rbrace$ for $V$ given as above.
\item Note that rational domains even satisfy the stronger universal property mentioned before: Let $V\subseteq \mt M(A)$ be a rational domain in $\mt M(B)$ be a $k'$-affinoid space for some complete isometric extension $k \subseteq k'$ with a morphism of $k$-Banach algebras $A \rightarrow B$ such that the corresponding morphism $\mt M(B) \rightarrow \mt M(A)$ factors through $V$. Then there exists a unique morphism of $k$-Banach algebras $A_V \rightarrow B$ that commutes with the other maps involved (cf. \cite{Be1} Remarks 2.2.2).
\end{enumerate}
}
\end{exmp}

The following theorem shows the tremendous importance of rational domains for the whole theory:   

\begin{theo}\label{theo. gerritzen Grauert} (Gerritzen-Grauert theorem)
Let $X=\mt M(A)$ be an affinoid space and $V \subseteq \mt M(A)$ be an affinoid domain in $X$. Then $V$ is a finite union of rational domains.
\end{theo}
\begin{proof}
\cite{Te2} §3.
\end{proof}

\subsection{Short excursion: stronger universal property for affinoid domains}

Most of the results in Berkovich's theory presented above are generalizations of Tate's theory of rigid analytic spaces. For the latter, the fundamental objects are \emph{Tate algebras} which are precisely the strictly $k$-affinoid algebras in Berkovich's theory. One important way to reduce assertions about general $k$-affinoid algebras to strictly $k$-affinoid algebras is given by the following propositions. At the end we will show that affinoid domains in $k$-affinoid spaces in fact satisfy a stronger universal property.
Note that $\hat{\otimes}$ denotes the completed tensor product (cf. \cite{Be1} 1.1).

\begin{prop}\label{prop. find k-affinoid algebra}
Let $A$ be a $k$-affinoid algebra. Then there exists a complete non-archimedean field extension $k \subseteq k'$ such that $A_{k'} = A \hat{\otimes}_k k'$ is a strictly $k'$-affinoid algebra. 
\end{prop}
\begin{proof}
\cite{Be1} §2.1.
\end{proof}

\begin{prop}\label{prop. exact sequences}
Let $(k',\vald')$ be a complete isometric extension of $k$. Then a sequence of bounded homomorphisms of $k$-Banach spaces 
\begin{gather*}		
L \rightarrow M \rightarrow N
\end{gather*}
is exact and admissible (that means, consisting of admissible maps) if and only if 
\begin{gather*}		
L\hat{\otimes}_k k'\rightarrow M\hat{\otimes}_k k' \rightarrow N\hat{\otimes}_k k'
\end{gather*}
is exact and admissible.
\end{prop}
\begin{proof}
\cite{FK} Lemma C.1.3. For the case where $k'$ is given as in \ref{prop. find k-affinoid algebra}, see \cite{Be1} Proposition 2.1.2 ii).
\end{proof}

With those assertions many problems about $k$-affinoid algebras can be reduced to questions about strictly $k$-affinoid algebras. One example is the following Berkovich version of Tate's acyclicity Theorem. 

\begin{theo}\label{theo. tate acyclicity} (Tate's Aclicity Theorem) 
Let $\lbrace V_i \rbrace _{i\in I}$ be a finite affinoid covering of a $k$-affinoid space $X=\mt M(A)$. For any finite Banach $A$-module $M$ the \v Cech complex is exact and admissible
\begin{gather*}
0 \rightarrow M \rightarrow \prod_i M \otimes_A A_{V_i} \rightarrow \prod_{i,j} M \otimes_A A_{V_i \cap V_j}\rightarrow ... 
\end{gather*}
\end{theo}
\begin{proof}
\cite{Be3} §1.2.
\end{proof}

\noindent Note that we do not have to use the completed tensor product in the theorem above, since we have $M \hat{\otimes}_A A_{V_i} \cong M \otimes_A A_{V_i}$ and $M \hat{\otimes}_A A_{V_i\cap V_j} \cong M \otimes_A A_{V_i\cap V_j}$ for all $i,j$ (cf. \cite{Be1} Proposition 2.1.10).

This result can be used to proof that affinoid domains of an affinoid space own a stronger universal property than that defined in \ref{def. affinoid domain}:

\begin{prop}\label{prop. universal property holds for all k- banach algebras}
Let $X=\mt M(A)$ be an affinoid space and $V \subseteq \mt M(A)$ be an affinoid domain in $X$. Then the universal property of $\mt M (A_V) \rightarrow \mt M(A)$ holds for all spectra of $k'$-affinoid algebras (where $k'$ is a complete isometric extension of $k$) and not merely for $k$-affinoid spaces.
\end{prop}
\begin{proof} (cf. \cite{Be1} Proposition 2.2.3 i))
Let $k'$ be a complete isometric extension of $k$, $B$ a $k'$-affinoid algebra and $\varphi: \mt M(B) \rightarrow \mt M(A)$ with $\varphi (\mt M(B)) \subseteq V $ a map induced by a bounded homomorphism of $k$-algebras $A \rightarrow B$. By the Gerritzen-Grauert theorem (\ref{theo. gerritzen Grauert}), we can find rational domains $V_1,...,V_n$ in $\mt M(A)$ that cover $V$. For $i=1,...,n$, we set $U_i:= \varphi^{-1}(V_i)$ which is a rational domain in $\mt M(B)$ (cf. \ref{example rational domain}). But since we know that rational domains satisfy the stronger universal property (again cf. \ref{example rational domain}), we get bounded homomorphisms of $k$-algebras $A_{V_i}\rightarrow B_{U_i}$ and $A_{V_i \cap V_j}\rightarrow B_{U_i \cap V_j}$ for all $i,j=1,...,n$. By Tate's Theorem \ref{theo. tate acyclicity}, we therefore have a commutative diagram
\begin{gather*}		
 \begin{xy}
  \xymatrix{
     0 \ar[r] &  A_V \ar[r] \ar@{.>}[d] & \prod_i A_{V_i} \ar[r] \ar[d]     &   \prod_{i,j} A_{V_i \cap V_j} \ar[d]   \\
      0 \ar[r] & B \ar[r]   & \prod_i B_{U_i} \ar[r]    &   \prod_{i,j} B_{U_i \cap U_j}.
  }
\end{xy}
\end{gather*} 
Since the horizontal lines are exact, we obtain a bounded $k$-algebra homomorphism $A_V \rightarrow B$ as desired. 
\end{proof}

\begin{rem}\label{rem. pre images of affinoid domains}
\emph{
Let $\varphi : Y:= \mt M(B) \rightarrow \mt M(A)=:X$ be a morphism of $k$-affinoid spaces and $V\subseteq X$ be an affinoid domain in $X$. Then $\varphi^{-1}(V)$ is an affinoid domain in $Y$ which is represented by the morphism of $k$-affinoid algebras $A \rightarrow B \hat{\otimes}_A A_V$ (cf. \cite{Be1} Remarks 2.2.2).\\ 
For example if $U$ is an affinoid domain in $X$ that is contained in the affinoid domain $V$, then $\psi^{-1}(U)$ is an affinoid domain in $\mt M(A_V)$ (here $\psi$ denotes the morphism of $k$-affinoid spaces $\mt M(A_V) \rightarrow \mt M(A)$). Moreover, in this case the respective $k$-affinoid spaces $\mt M(A_U)$ and $\mt M((A_V)_{\psi^{-1}(U)})$ are isomorphic since they satisfy the same universal property. In particular, we get an isomorphism of $k$-affinoid algebras $A_U \cong (A_V)_{\psi^{-1}(U)}$.
}
\end{rem}

\subsection{k-analytic Berkovich spaces}

In this section we glue affinoid spaces together to obtain $k$-analytic Berkovich spaces. Since this will be done via affinoid domain embeddings which are closed topological embeddings the character of this process is different from the procedure we used for general adic spaces.

\begin{defn}\label{def. net and quasi net}
\emph{Let $X$ be a topological space and $\tau$ be a set of subsets of $X$.
\begin{enumerate}
\item $\tau$ is said to be a \emph{quasi-net on $X$} if for every $x \in X$ there are finitely many elements of $\tau$ such that their union is a neighbourhood of $x$ and $x$ is contained in their intersection.
\item $\tau$ is called a \emph{net on $X$} if it is a quasi-net on $X$ and for each $U,V \in \tau$ the set $\tau |_{U\cap V}:= \lbrace W \in \tau \ | \ W \subseteq U \cap V \rbrace$ is a quasi-net on $U\cap V$.
\end{enumerate}
}
\end{defn}

\begin{rem}\label{rem. nets as categories}
\emph{For a topological space $X$ with net $\tau$ we will consider $\tau$ as a category in the natural way (morphisms are inclusion maps). Moreover, we denote the canonical functor from $\tau$ to the category of topological spaces by $\mathcal{T}$.}
\end{rem}

\begin{defn}\label{def. k-analytic berkovich spaces}
\emph{Let $X$ be a locally Hausdorff space and $\tau$ be a net on $X$ consisting of compact subsets. A \emph{$k$-affinoid atlas of $X$ with net $\tau$} (denoted by $\mt A$) is given by:
\begin{enumerate}
\item A $k$-affinoid algebra $A_U $ and a homeomorphism $U \rightarrow \mt M(A_U)$ for each $U \in \tau$;
\item For each pair $U,U' \in \tau$ with $ U \subseteq U'$ a bounded $k$-algebra homomorphism $\rho_U^{U'}: A_{U'}\rightarrow A_{U}$ such that the corresponding morphism of $k$-affinoid spaces $\mt M(A_U) \rightarrow \mt M(A_{U'})$ is an affinoid domain embedding.\\
\\
 \noindent These data should satisfy the following cocycle condition:
\item[iii)] For any triple $U ,U' ,U'' \in  \tau$ such that $U \subseteq U' \subseteq U''$ the equation $\rho_U^{U''}=\rho^{U'}_U \circ \rho_{U'}^{U''}$ holds.
\end{enumerate}
A \emph{$k$-analytic Berkovich space} is a triple $(X,\mathcal{A},\tau)$ where $X$ is a locally Hausdorff space, $\tau$ is a net on $X$ that consists of compact subsets and $\mathcal{A}$ is a $k$-affinoid atlas of $X$ with net $\tau$.\\
One can similarly define strictly $k$-analytic Berkovich spaces by running through all definitions above replacing "$k$-affinoid" by "strictly $k$-affinoid". 
}
\end{defn}

The following definition provides a naive notion of morphisms between $k$-affinoid Berkovich spaces. As we will see in the remark below, it is not sufficient for the theory.

\begin{defn}\label{def. strong morphisms of k-analytic spaces}
\emph{
Let $(X,\mathcal{A},\tau)$ and $(X',\mathcal{A}',\tau')$ be $k$-analytic Berkovich spaces.
A \emph{strong morphism} $(X,\mathcal{A},\tau)\rightarrow (X',\mathcal{A}',\tau')$ consists of:
\begin{enumerate}
\item a continuous map $\varphi : X \rightarrow X'$ such that for each $U \in \tau$ there is an $U' \in \tau'$ with $\varphi(U)\subseteq U'$;
\item a bounded $k$-algebra homomorphism $\phi_{U/U'}: A'_{U'}\rightarrow A_U$ such that the induced map $\mathcal{M}(A_U)\rightarrow \mathcal{M}(A'_{U'})$ is identified with $\varphi |_U : U \rightarrow U'$ via $\mathcal{M}(A_U)\cong U$ and $\mathcal{M}(A'_{U'})\cong U'$ for each pair $(U,U')\in \tau \times \tau'$ with $\varphi(U)\subseteq U'$.
\end{enumerate}
}
\end{defn}

\begin{rem}\label{rem. affinoid domains as k-analytic spaces}
\emph{Let $X=\mt M(A)$ be a $k$-affinoid space (resp. strictly $k$-affinoid space). Choosing $\tau = \lbrace X \rbrace$ and defining a $k$-affinoid atlas $\mt A$ by $X\mapsto A$, we obtain a $k$-analytic space (resp. strictly $k$-analytic space) $(X,\mt A, \tau)$. On the other hand, we can also choose $\tau'$ to be the set containing all affinoid domains (resp. strictly $k$-affinoid domains) of $X$ and $\mt A'$ to be the canonical $k$-affinoid atlas (resp. strictly $k$-affinoid atlas) sending $U\in \tau'$ to $A_U$ where $A_U$ comes from the definition of affinoid domains. So we obtain a further $k$-analytic space (resp. strictly $k$-analytic space) $(X,\mt A',\tau')$ with ground space $X$.\\
Note that both spaces are completely determined by $A$, and hence they contain the same amount of information although $(X,\mt A',\tau')$ seems to be \emph{finer}  in a certain sense. Since we do not want to distinguish between analytic spaces that just differ by a certain refinement, we need a notion of morphisms that respects this requirement.
}
\end{rem}

Making this more precise, one defines a \emph{maximal $k$-analytic atlas} $(\widehat{\mt A}, \widehat{\tau})$ for a $k$-analytic Berkovich space $(X, \mt A, \tau)$ with the property that $(\widehat{\widehat{\mt A}}, \widehat{\widehat{\tau}})$ equals to $(\widehat{\mt A}, \widehat{\tau})$.  Morphisms in the category of $k$-analytic Berkovich spaces will then be defined in such a way that the canonical strong morphism $(X, \mt A, \tau) \rightarrow (X, \widehat{\mt A}, \widehat{\tau})$ induces an isomorphism.

\begin{rem}\label{rem. refinement of nets in Berkovich spaces: include all affinoids}
\emph{
Let $(X, \mt A, \tau)$ be a $k$-affinoid Berkovich space.
\begin{enumerate}
\item If $W$ is an affinoid domain in some $U\in \tau$, then it is an affinoid domain in any $V\in \tau$ that contains $W$. Indeed, by applying the universal property of affinoid domains, this assertion is clear if $X= \mt M(A)$ is $k$-affinoid space. For the general case see \cite{Be3} Lemma 1.2.5.
\item We define $\overline{\tau}$ to be the set of all $W's$ from i). Then $\overline{\tau}$ is a net on $X$ and there exists a unique extension of $\mt A$ to $\overline{\tau}$ which we denote by $\overline{\mt A}$.
\item Let $\varphi : (X,\mathcal{A},\tau) \rightarrow (Y,\mathcal{A'},\tau')$ be a strong morphism of $k$-analytic Berkovich spaces. Then $\varphi$ can uniquely be extended to a strong morphism $(X,\overline{\mathcal{A}},\overline{\tau}) \rightarrow (Y,\overline{\mathcal{A'}},\overline{\tau'})$ (cf. \cite{Be3} Proposition 1.2.8)
\item Let $\widehat{\tau}$ be the set containing all compact Hausdorff subsets $W\subseteq X$ such that $W$ admits a finite cover $(W_i)_{i \in I}$ with $W_i\in \overline{\tau}$ for all $i \in I$ and such that the following properties are satisfied:
\begin{enumerate}
\item[a)] $W_i \cap W_j \in \overline{\tau}$ and natural the map $\overline{\mt A}_{W_i}\widehat{\otimes}_k \overline{\mt A}_{W_j} \rightarrow \overline{\mt A}_{W_i\cap W_j}$ is an admissible surjection for all $i,j \in I$.
\item[c)] The $k$-Banach algebra $\widehat{\mt A}_{\lbrace W_i \rbrace_{i\in I}}:= \text{ker}(\prod_i \overline{\mt A}_{W_i} \rightarrow \prod_{i,j} \overline{ \mt A}_{W_i \cap W_j})$ is $k$-affinoid and the obtained map of sets $W\rightarrow \mt M(\widehat{\mt A}_{\lbrace W_i \rbrace})$ is a homeomorphism .
\end{enumerate}
\end{enumerate}
}
\end{rem}

One shows that $(X,\widehat{\mt A}, \widehat{\tau})$ in fact gives a $k$-analytic Berkovich space (cf. \cite{Be3} Proposition 1.2.13). One also uses the definition above to endow $X$ with a Grothendieck topology and a structure sheaf which we will not do in this work (for more details see \cite{Be3} §1.3).

Now we finally approach the right definition of a morphism. Beginning with the category of $k$-analytic Berkovich spaces and strong morphisms, we will define a category of fractions where we localize with respect to the following special strong morphisms:

\begin{defn}\label{def. quasi isomorphisms}
\emph{A strong morphism  $(\varphi,(\phi_{U/U'})):(X,\mathcal{A},\tau)\rightarrow (X',\mathcal{A}',\tau')$ of $k$-analytic spaces is said to be a \emph{quasi-isomorphism} if $\varphi$ is a homeomorphism and all $\phi_{U/U'}$ are affinoid domain embeddings.
}
\end{defn}

\begin{rem}\label{rem. calculus of fractions}
\emph{
Let $(X,\mathcal{A},\tau),(X,\mathcal{A}',\tau')$ be $k$-analytic spaces such that $\tau$ is a subset of $\tau'$ and $\mt A = \mt A' |_\tau$ (i.e. $ A_U =  A'_U$ for all $U \in \tau$). Then the identity $X \rightarrow X$ obviously induces a quasi-isomorphism $(X,\mathcal{A},\tau)\rightarrow (X,\mathcal{A}',\tau')$. In particular  $(X,\mathcal{A},\tau) \rightarrow (X, \widehat{\mt A}, \widehat{\tau})$ becomes a quasi-isomorphism, as desired. 
}
\end{rem}

Note that the system of quasi-isomorphisms admits \emph{calculus of fractions} in the category of $k$-analytic spaces with strong morphisms. Hence the following definition makes sense (cf. \cite{Be3} Proposition 1.2.10).

\begin{defn}\label{def. category of k-analytic spaces}
\emph{The \emph{category of $k$-analytic Berkovich spaces} is the quotient category of the category of $k$-analytic spaces with strong morphisms by the system of quasi-isomorphisms. I.e. a morphism of $k$-analytic spaces  $(X,\mathcal{A}_X,\tau_X)\rightarrow (Y,\mathcal{A}_Y,\tau_Y)$ is given by the equivalence class of a diagram (which we sometimes call a \emph{span})  
\begin{gather*}
(X,\mathcal{A}_X,\tau_X)\leftarrow (X',\mathcal{A}_{X'}, \tau_{X'})\rightarrow (Y,\mathcal{A}_Y,\tau_Y)
\end{gather*}
such that the left arrow is a quasi-isomorphism and the right arrow is a strong morphism.\\
The associated equivalence relation is defined as follows: Let $a,a',a''$ and $b$ be $k$-analytic Berkovich spaces and $v:a \rightarrow a'$, $w:a \rightarrow a''$, $f: a' \rightarrow b$ and $g: a'' \rightarrow b $ be strong morphisms such that $v$ and $w$ are quasi-isomorphisms.
$a \xleftarrow{v} a' \xrightarrow{f} b$ is equivalent to $a \xleftarrow{w} a'' \xrightarrow{g} b$ if there exists a $k$-analytic Berkovich space $\overline{a}$ and strong morphisms $s: \overline{a} \rightarrow a'$ and $t: \overline{a} \rightarrow a''$ such that $f \circ s = g \circ t$ and $v \circ s= w \circ t$ are quasi-isomorphism. This definition is illustrated by the diagram
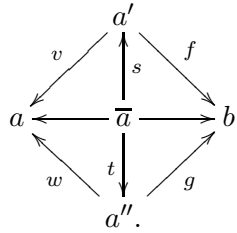
\begin{gather*}
\begin{xy}
  \xymatrix{
		& a'  \ar[ld]_{v} \ar[rd]^{f}  \\
		a & \overline{a}\ar[u]_s \ar[l] \ar[r] \ar[d]_t & b\\
		& a''  \ar[lu]^{w} \ar[ru]_{g}.\\
 }
\end{xy}
\end{gather*}
We have the following definition of composition of morphisms:
Given two morphisms of $k$-analytic Berkovich spaces $f:a\rightarrow b$ and $g:b\rightarrow c$ represented by $a\xleftarrow{v} a' \xrightarrow{f} b$ and $b \xleftarrow{u} b' \xrightarrow{h} c$. Then there exist strong morphisms of $k$-analytic Berkovich spaces $z:d\rightarrow a'$ and $k:d \rightarrow b'$ such that $z$ is a quasi-isomorphism and $f \circ z =u\circ k$ (this holds since the system of quasi-isomorphisms admits calculus of fractions). We define $g\circ f$ to be the equivalence class of $a\xleftarrow{v\circ z} d \xrightarrow{h\circ k} c$.\\
For more details on \emph{calculus of fractions} and in particular arguments why the notions above are well-defined, see \cite{GZ} §2.
}
\end{defn}

\begin{rem}\label{rem. category of strictly k-analytic spaces}
\emph{
As hinted for objects of the category defined above, one similarly defines the \emph{category of strictly $k$-analytic Berkovich spaces} by using only \emph{strictly $k$-affinoid spaces} and \emph{strictly affinoid domains} in those definitions. It is not trivial that any morphism of $k$-analytic Berkovich spaces $f:X\rightarrow Y$ between strictly $k$-affinoid Berkovich spaces is also a morphism of strictly $k$-affinoid Berkovich spaces:\\ Given a realization of $f$, $X \leftarrow X' \rightarrow Y$ (where the left arrow is a quasi-isomorphism and the right arrow is a strong morphism) it is a priori not clear if the $k$-analytic space $X'$ can be chosen to be strict.\\
This problem was solved by Temkin within the following theorem: 
}
\end{rem}

\begin{theo}\label{theo. strict fully faithful to analytic}
The canonical inclusion functor from the category of strictly $k$-analytic Berkovich spaces to the category of $k$-analytic Berkovich spaces is fully faithful.
\end{theo}
\begin{proof}
\cite{Te3} Corollary 4.10.
\end{proof}

\begin{rem}\label{rem. being strict no property}
\emph{
The theorem above also ensures that we can regard the category of strictly $k$-analytic Berkovich spaces as a subcategory of the category of $k$-analytic Berkovich spaces. So a $k$-analytic Berkovich space is said to be \emph{strict} if it is contained in the essential image of the inclusion functor.
}
\end{rem}

\section{The affine case}

The main goal of this paper is to prove an equivalence of categories between strictly $k$-analytic Berkovich spaces and taut adic spaces that are locally of finite type over $k$. Hence a natural approach is to consider the affine objects first and then go up to the global level. In the first paragraph of this section we show that the notions of strictly $k$-affinoid algebras and Huber pairs that are topologically of finite type over $k$ essentially coincide and hence the equivalence of categories in the affine case easily follows. However, it is not obvious how this equivalence can be extended to the global level since the respective glueing process differs significantly. Therefore it is necessary to investigate the connection between the affine objects on a geometric level, which will be done in the second subsection of this paragraph.

\subsection{Equivalence of categories in the affine case}

\begin{rem}\label{rem. k-affinoid algebra considered as a Huber pair }
\emph{Let $A$ be a strictly $k$-affinoid algebra and let $c \in k$ be an arbitrary topologically nilpotent unit. Then $A$ can be considered as a Huber ring with pair of definition $(A^o,c \cdot A^o)$ and this construction does not depend on the choice of $c \in k^o$. As mentioned before one can always choose $A^o$ as a ring of integral elements of $A$ to obtain a Huber pair $(A,A^o)$. The following propositions deal with this mapping $A \mapsto (A,A^o)$ and show that it induces an equivalence of categories.
}
\end{rem}

\begin{lem}\label{lem. unique topological of finite type}
Let $A$ be a strictly $k$-affinoid algebra. Then there exists a unique ring of integral elements $A^+$ of $A$ such that $k^o \subseteq A^+$ and such that $(A,A^+)$ is topologically of finite type over $k$, namely $A^+=A^o$.
\end{lem}
\begin{proof}
\cite{Hu2} Lemma 4.4.
\end{proof}

\begin{prop}\label{prop. equivalence of cat. k-affinoid algebras, complete Huber pairs locally of finite type /k}
We have a functor $A \mapsto (A,A^o)$ from the category of strictly $k$-affinoid algebras with bounded morphisms to the category of complete Huber pairs that are topologically of finite type over $k$ and it  is an equivalence of categories. 
\end{prop}
\begin{proof}
A bounded homomorphism of strictly $k$-affinoid algebras $\varphi : A \rightarrow B$ is obviously continuous  and maps power-bounded elements to power-bounded elements by \ref{rem. bounded brought together}. On the other hand if we have a continuous morphism of Huber pairs $\varphi :(A,A^o)\rightarrow (B,B^o)$ where $A$ and $B$ are strictly $k$-affinoid algebras, then it is bounded by \ref{lem. bounded means continuous}.
Hence we have a fully faithful functor from the category of strictly $k$-affinoid algebras to the category of complete Huber pairs, $A \mapsto (A,A^o)$. The uniqueness assertion of \ref{lem. unique topological of finite type} and the fact that the underlying Huber ring of a Huber pair that is locally of finite type over $k$ is a strictly $k$-affinoid algebra (cf. \ref{rem. huber pairs of finite type equal k affinoid }) ensures that the image of $A \mapsto (A,A^o)$ precisely consists of those complete Huber pairs that are topologically of finite type over $k$.
\end{proof}

\begin{prop}\label{prop. equivalence affine adic spaces locally of finite type over k strict k-affinoid spaces}
The functor $\mt M(A) \mapsto \Spa(A,A^o)$ from the category of strictly $k$-affinoid spaces to the category of affinoid adic spaces, that are locally of finite type over $k$, is an equivalence of categories.
\end{prop}
\begin{proof}
The functor $(A,A^+)\mapsto \Spa(A,A^+)$ from the category of complete Huber pairs that are topologically of finite type over $k$ to the category of affinoid adic spaces induces a categorial equivalence onto the category of affinoid adic spaces locally of finite type over $k$ (cf. \ref{theo. fully faithful functor} and \ref{prop: locally of finite type second prop}). Since the category of strictly $k$-analytic spaces is by definition equivalent to the category of strictly $k$-affinoid algebras, the claim follows from (\ref{prop. equivalence of cat. k-affinoid algebras, complete Huber pairs locally of finite type /k}).
\end{proof}

\subsection{A continuous map $\Spa(A,A^o) \rightarrow \mt M(A)$}

The previous section provided an equivalence on the affine level of adic and Berkovich spaces. Nevertheless, the equivalence on the global level is not obvious since there is a significant difference in the glueing processes of the respective settings. In both cases the rational subsets play a tremendous role. Whereas rational subsets form an open basis of the topology of an adic space, rational domains in $k$-analytic Berkovich spaces form fundamental systems of closed neighbourhoods for any point. Hence it is necessary to take a closer look at the geometric spaces $\Spa(A,A^o)$ and $\mt M(A)$ and their connections. More precisely, in this subsection we will construct a map $q:\Spa(A,A^o) \rightarrow \mt M(A)$ which carries properties from the adic setting to the Berkovich world and the other way around. In the following paragraph (§5), we will focus on how this map can be generalized to the global level and which information is transferred.\\

First we start with an easy lemma that will be applied various times in this section:

\begin{lem}\label{lem. power bounded topologically nilpotent}
Let $x: A \rightarrow \Gamma \cup \lbrace 0 \rbrace$ be a valuation on a ring $A$ which is endowed with the $x$-topology. Let $f \in A$.
\begin{enumerate}
\item If $f$ is topologically nilpotent, then $x(f)<1$;
\item if $x(f) \leq 1$, then $f$ is power-bounded;
\item if $x$ is of rank $1$ and $x(f)<1$ then $f$ is topologically nilpotent;
\item if $x$ is of rank $1$, $A$ has a nilpotent unit (e.g. $A$ is a field) and $f$ is power-bounded, then $x(f)\leq 1$.
\end{enumerate}
\end{lem}
\begin{proof}
Let $B:= \lbrace f^n \ | \ n\in \mathbb{N} \rbrace$ and $U_{\gamma}:= \lbrace a \in A \ | \ x(a)<\gamma\rbrace$ for $\gamma \in \Gamma$.
\begin{enumerate}
\item Assume that $f$ is topologically nilpotent and $x(f)\geq 1$. Then $x(f^n)\geq 1 $ for all $n \in \mathbb{N}$. Hence $B\cap U_1 = \emptyset$ which is a contradiction.
\item Let $x(f) \leq 1$, then $U_\gamma \cdot B \subseteq U_\gamma$ for all $\gamma \in \Gamma$. Hence $f$ is power-bounded.
\item Let $x(f)<1$ and $\gamma \in \Gamma$ arbitrary. Since there are no proper non-trivial convex subgroups of $\Gamma$, there is an $n \in \mathbb{N}$ such that $x(f)^n<\gamma$, which means
$f^n \in U_\gamma$. So in this case we have $A^{00}=\lbrace a \in A \ | \ x(a) < 1 \rbrace$.
\item Assume that we are in the situation of $iv)$ but $x(f)>1$. Let $\gamma \in \Gamma$ be arbitrary and let $\varpi$ be a topologically nilpotent unit in $A$. Then there exists an $n \in N$ such that $x(\varpi)^{n}< \gamma$, that means $\varpi^n \in U_\gamma$. Now choose $k \in \mathbb{N}$ with $x(f^k)>x(\varpi^n)$. Then we have $1<x(\varpi^n)x(f^{k})$ and in particular $U_\gamma \cdot B \nsubseteq U_1$ for all $\gamma\in \Gamma$. But this is a contradiction since we assumed that $B$ is bounded. Therefore we conclude $A^{0}=\lbrace a \in A \ | \ x(a) \leq 1 \rbrace$.
\end{enumerate}
\end{proof}

In the following, we want to consider the Berkovich spectrum of a strictly $k$-affinoid algebra $A$ to be a subset of the adic spectrum of its associated Huber pair $(A,A^o)$. Therefore we have to clarify some details since the definitions differ slightly.

\begin{lem}\label{lem. valuations are bounded}
Let $A$ be a strictly $k$-affinoid algebra and fix a norm $\normd$ on $A$. Let $v:A\rightarrow \mathbb{R}^{\geq 0}$ be a continuous valuation of $A$ such that there exists $c \in k^\times$ with $\val c <1$ and $v(c)=\val c$. Then $v$ is bounded, i.e. there exists $C>0$ such that $v(a) \leq C \cdot \norm a$ for all $a \in A$.
\end{lem}
\begin{proof}
For $a \in A$ with $\norm a < 1$, we have $a \in A^{oo}$ and hence $v(a)<1$ by \ref{lem. power bounded topologically nilpotent}. So let $a \in A$ be arbitrary and choose $n\in \mathbb{N}$ such that 
\begin{gather*}		
\val c^n \norm a <1 \leq \val c^{n-1} \norm a.
\end{gather*}
Then $\norm {c^n a}<1$ and hence $v(c^na)<1$. But then we get 
\begin{gather*}		
v(a)<v(c)^{-n}=\val c^{-n}= \val c^{-1} \val c^{-n+1} \leq \val c^{-1} \norm a.
\end{gather*}
Hence $v$ is bounded.
\end{proof}

\begin{lem}\label{lem. unique order preserving hom to R}
Let $\Gamma$ be a totally ordered group of rank $1$. Let $j,j':\Gamma\rightarrow \mathbb{R}^{>0}$ be two injective morphisms of totally ordered groups. Then we have $j'=j^e$ for a unique $e>0$.
\end{lem}
\begin{proof}
Let $\gar\in \Gamma_{<1}$. Then there is a unique $d>0$ such that $j(\gar)^d=j'(\gar)$. By replacing $j$ with $j^d$ we can therefore assume that $j(\gar)=j'(\gar)$. Now let $\gamma \in \Gamma_{<1}$ be arbitrary. We show that $j(\gamma)=j'(\gamma)$ by an approximation argument, which implies the claim for all $\gamma \in \Gamma$ since $j$ and $j'$ are group homomorphisms. Now let $m\in \mathbb{N}$ be arbitrary and choose $n\in \mathbb{N}_0$ (depending on $m$) such that $\gar^{n+1}\leq \gamma ^m \leq \gar^n$. Note that such an $n$ exists since $\Gamma$ is of rank $1$ and hence the convex subgroup \begin{gather*}		
\Delta := \lbrace a \in \Gamma \ | \ \exists \ l,r \in \mathbb{Z} \ \text{such that} \ \gar^l\leq a \leq \gar ^r \rbrace
\end{gather*}
has to be equal to $\Gamma$. We get $j(\gar)^{\frac{n+1}{m}}\leq j(\gamma), j'(\gamma) \leq j(\gar)^{\frac{n}{m}}$ which leads to $\frac{n+1}{m}\leq \log_{j(\gar)}j(\gamma),\log_{j'(\gar)}j'(\gamma)\leq \frac{n}{m}$. Since $j(\gar)=j'(\gar)$, this finishes the prove by taking $m\rightarrow \infty$. 
\end{proof}

\begin{prop}\label{prop. unique representant}
Let $A$ be a strictly $k$-affinoid algebra with fixed norm $\normd$ and let $v:A\rightarrow \Gamma \cup \lbrace 0 \rbrace$ be a representative of a rank $1$ point of $\Spa(A,A^o)$. Then there exists a unique element of $\mt M(A)$ that is equivalent to $v$ (as a valuation on $A$).
\end{prop}
\begin{proof}
First note that all elements of $\mt M(A)$ are representatives of rank $1$ points of $\Spa(A,A^o)$. Indeed they are continuous since they are bounded. Moreover, as seen in \ref{rem morphisms of spectra} they extend the valuation on $k$ and hence satisfy the non-archimedean triangle inequality by \ref{rem normal triangle inequality implies ultrametric}. Finally since they are multiplicative and bounded they send power-bounded elements of $A^o$ to $[0,1]$.\\
Choose an injective homomorphism of totally ordered groups $j':\Gamma \rightarrow \mathbb{R}^{>0}$ (which exists since $\Gamma$ has rank $1$, cf. \cite{We1} Proposition 1.14.) and a $c \in k^\times$ such that $\val c <1$. Since $c$ is topologically nilpotent, we have $j'(v(c))<1$ and hence there is an $e>0$ such that $\val c = j'(v(c))^e$. We set $j:=(j')^e$. $j \circ v $ (and any power of $j$ composed with $v$) is obviously a multiplicative seminorm on $A$. Moreover, it is bounded by \ref{lem. valuations are bounded} (which implies that $\val a = j (v(a))$ for all $a\in k$ by \ref{rem. represetation of k affinoid algebras is not unique}) and hence $\tilde{j}\circ v$ is an element of $\mt M(A)$ which is equivalent to $v$ as a valuation (here $\tilde{j}$ denotes the extension of $j$ to the monoid $\Gamma \cup \lbrace 0 \rbrace$).  Assume that there is another element $w:A\rightarrow \mathbb{R}^{\geq 0}$ of $\mt M(A)$ that is equivalent to $v$ as a valuation (in particular $w$ restricted to $k$ is $\vald$). Let $V \subseteq \mathbb{R}^{>0}$ be the valuation group of $\tilde{j} \circ v$ and $W \subseteq \mathbb{R}^{>0}$ the valuation group of $w$. Then we have a commutative diagram
\begin{gather*}		
\begin{xy}
  \xymatrix{
      A \ar[r]^w   \ar[d]^v  &   W   \\
     \Gamma \ar[r]_j           &   V \ar[u]_f   
  }
\end{xy}
\end{gather*}
where $f$ is an isomorphism of totally ordered groups. This means that $f\circ j$ is an injective morphism of totally ordered groups as well and hence a power of $j$ by \ref{lem. unique order preserving hom to R} but this implies $f\circ j =j$ since both maps coincide on elements of $\Gamma$ that come from $k$. So we have $w=v \circ j$.
\end{proof}

\begin{prop}\label{prop. map: spa a to m(a)}
Let $A$ be a strictly $k$-affinoid algebra. We have a well defined continuous surjection $q:\Spa(A,A^o)\rightarrow \mt M(A)$ sending an element of $\Spa(A,A^o)$ to the unique representative from \ref{prop. unique representant} of its unique maximal generization. Moreover, this map makes $\mt M(A)$ the maximal $T_1$-quotient of $\Spa(A,A^o)$ (that means that any continuous map from $\Spa(A,A^o)$ to a $T_1$-space factors uniquely through $\mt M(A)$). 
\end{prop}
\begin{proof}
Set $X:= \Spa(A,A^o)$. Since $A$ contains a topologically nilpotent unit, $X$ is analytic in the sense of \ref{def. analytic adic spaces}. Therefore for $x \in X$ there exists at most one maximal generization by \ref{prop: properties of analytic adic spaces} i).  Such a generization always exists since $X$ is sober (cf. \ref{rem. zorn and co}). By \ref{prop: properties of analytic adic spaces} ii) the maximal generizations in $X$ are precisely the rank $1$ valuations.  Hence by \ref{prop. unique representant} we have a well defined map $q: X \rightarrow \mt M(A)$. $q$ restricted to the rank $1$ points of $X$ provides a bijection. Now we approach the continuity of $q$.\\
By \ref{lem. topology of berkovich spectrum} sets of the form $V_{f,g}:=\lbrace z \in \mt M(A) \ | \ \val f_z < \val g_z \rbrace $, where $f,g \in A$ form a subbasis of the topology of $\mt M(A)$. Fix $f,g \in A$ and set $V:=V_{f,g}$. So it is enough to show that $q^{-1}(V)$ is open in $X$.
For $x \in X$ let $\eta_x \in X$ denote the unique maximal generization of $x$. Then $q^{-1}(V)= \lbrace x \in X \ | \ \eta_x(f)<\eta_x(g) \rbrace$. Choose $c \in k$ such that $0<\val c < 1$. Let $x_0  \in q^{-1}(V)$ (in particular we have $\eta_{x_0}(g) \neq 0$) and consider the residue field $\kappa:=\kappa(x_0):= \text{Frac}(A/\supp (x_0))$ (which is equal to $\kappa(\eta_{x_0})$ since $x_0$ and $\eta_{x_0}$ have the same support). Let $\tilde{x_0}$ and $\tilde{\eta_{x_0}}$ denote the unique extensions of $x_0$ and of $\eta_{x_0}$ to $\kappa$ respectively. Then the $\tilde{x_0}$ and the $\tilde{\eta_{x_0}}$ topology on $\kappa$ coincide (cf. \ref{def. defs for valuations} \emph{vi)}). 
Let $g(x_0)$ and $f(x_0)$ denote the respective images of $f$  and $g$ in $\kappa$. Then $g(x_0)\neq 0$ since $\tilde{\eta_{x_0}}(g(x))=\eta_{x_0}(g)\neq 0$ and hence $f(x)/g(x)$ is an element of $\kappa$. Since $\eta_{x_0}$ is a rank $1$ valuation, $\tilde{\eta_{x_0}}$ is a rank $1$ valuation of $\kappa$. Therefore from $\tilde{\eta_{x_0}}(f(x)/g(x))<1$ we get that $f(x)/g(x)$ is topologically nilpotent for the $\tilde{\eta_{x_0}}$-topology by \ref{lem. power bounded topologically nilpotent} iii) and hence it is also topologically nilpotent for the $\tilde{x_0}$-topology on $\kappa$.
Let $A_{\tilde{x_0}} \subseteq \kappa$ be the valuation ring of $\tilde{x_0}$ which is open (cf. \ref{rem. extension of valuations to field}). Then $c \cdot A_{\tilde{x}}$ is an open neighbourhood of $0$ in $\kappa$ and hence we can find $n \in \mathbb{N}$ such that $(f(x)/g(x))^n \in c \cdot A_{\tilde{x_0}}$. This means that $\tilde{x_0}((1/c)(f(x)/g(x))^n)\leq 1$ and hence $x_0(f)^n \leq x_0(c)x_0(g)^n$ in the valuation group of $x_0$. The set $\Omega := \lbrace x \in X \ | \ x(f^n)\leq x(cg^n)\neq 0 \rbrace$ is an open neighbourhood of $x_0$ in $X$ that is contained in $q^{-1}(V)$ and hence continuity of $q$ follows. Indeed, $\Omega$ is clearly open in $X$ and contains $x_0$ by construction. So let $x \in \Omega$. As above, since $x(g)\neq 0$, $g(x)$, the image of $g$ in $\kappa(x)$ is not zero and hence $f(x)^n/cg(x)^n$ is an element of $\kappa(x)$ which is power-bounded with respect to the $\tilde{x}$ topology on $\kappa(x)$ by \ref{lem. power bounded topologically nilpotent} ii). Again it is also power-bounded for the $\tilde{\eta_x}$-topology on $\kappa(x)$ and since $\tilde{\eta_x}$ is a rank $1$ valuation, we have $\tilde{\eta_x}(f(x)^n/cg(x)^n)\leq 1$ by \ref{lem. power bounded topologically nilpotent} iv). Therefore $\eta_x(f)^n\leq \eta_x(c)\eta_x(g)^n$ in the valuation group of $x$. But since $|c|<1$ and $\eta_x(c)$ has a representative in $\mt M(A)$, we have $\eta_x(c)< 1$. So $\eta_x(f)<\eta_x(g)$ follows and we conclude $\Omega \subseteq q^{-1}(V)$ as desired.

Let $f:X\rightarrow Y$ be a continuous map from $X$ to a $T_1$-space $Y$. Let $x,y \in X$ such that $q(x)=q(y)$. We want to show $f(x)=f(y)$ to get a unique set theoretical map $g: \mt M(A) \rightarrow Y$ such that $f=g \circ q$. We can assume that $y=\eta_x$ is the unique maximal generization of $x$ in $X$. But since $Y$ is a $T_1$-space, $f^{-1}(f(\eta_x))$ is a closed subset of $X$ containing $\eta_x$, hence $f^{-1}(f(\eta_x))$ contains $\overline{\lbrace \eta_x \rbrace}$ which in turn contains $x$. Therefore we have $f(x)=f(\eta_x)$. Moreover, $q$ is closed: Let $V \subseteq X$ be a closed subset, then it is quasi-compact by the quasi-compactness of $X$. Since $q$ is continuous, $q(V)$ is quasi-compact in $\mt M(A)$ and therefore closed, since $\mt M(A)$ is Hausdorff. $q$ is closed and  hence $\mt M(A)$ is endowed with the quotient topology of $X$ induced by $q$. Combined with what we said before, this means that $\mt M(A)$ is a maximal $T_1$-quotient of $X$.
\end{proof}

\begin{rem}\label{rem. section is not continuous in general}
\emph{Note that the natural section of $q$ sending an element of $\mt M(A)$ to its corresponding equivalence class in $\Spa(A,A^o)$ is not continuous in general (cf. \cite{Co2} Example 11.3.18).
}
\end{rem}

\section{Valuative spaces}

At first we recall some basic definitions of topological spaces and introduce \emph{valuative spaces} following Fujiwara and Kato in \cite{FK}. With the help of such spaces one can study the topological properties of certain adic spaces and their connections to Berkovich spaces.

\subsection{Valuative spaces}

\begin{defn}
\emph{A topological space $X$ is said to be \emph{valuative} if the following conditions are satisfied:
	\begin{enumerate}
	\item $X$ is locally coherent and sober;
	\item for every $x \in X$ the partially ordered set of generizations of $x$, denoted by $G_x$, is totally ordered.
\end{enumerate}
}
\end{defn}

\begin{rem}\label{rem. zorn and co}
\emph{ 
\begin{enumerate}
\item
Using Zorn's lemma, one can show that every point of a topological space $X$ is contained in an irreducible component. Hence any $x \in X$ admits a maximal generization. Indeed a maximal generization of a point $x \in X$ is given by the generic point of an irreducible component containing $x$.
\item Let $X$ be a valuative space. Since $G_x$ is totally ordered for all $x \in X$, the maximal generization of $x$ is unique which means that each $x \in X$ is contained in a unique irreducible component.  
\item Any open subset $U$ of a valuative space $X$ again is a valuative space. Moreover, the maximal points of such an open subset $U$ of $X$ are those of $X$ which are contained in $U$.
\end{enumerate}
}
\end{rem}

\begin{exmp}\label{exmp: Spa A is a valuative space}
\emph{This example will make clear how the terminology of valuative spaces is connected to the theory of adic spaces:
\begin{enumerate}
\item Let $(A,A^+)$ be a complete Huber pair such that $A$ contains a topologically nilpotent unit (e.g. $A$ is a strictly $k$-affinoid algebra and $A^+=A^{o}$). Then $\Spa(A,A^+)$ is an analytic adic space and in particular the underlying topological space is a quasi-compact valuative space (cf. \ref{prop: properties of analytic adic spaces}).
\item More generally, let $X$ be an analytic adic space. As above  \ref{prop: properties of analytic adic spaces} ensures that for each $x \in X$ the set of generizations is totally ordered and hence the underlying topological space of $X$ is a valuative space. 
\end{enumerate}
}
\end{exmp}

\subsection{Separated quotients}
Let $A$ be a strictly $k$-affinoid algebra. In \ref{prop. map: spa a to m(a)} we have seen how the topologies of the associated affinoid adic space $\Spa(A,A^o)$ and the associated Berkovich spectrum $\mt M(A)$ are connected. The main idea, to get a similar result for the connection between more general adic spaces and Berkovich spaces, is to generalize the map $q:\Spa(A,A^o) \rightarrow \mt M(A)$. We will first introduce this map in the setting of general valuative spaces and point out some of its properties which are most important for this work.
\begin{defn}\label{defn: separation map}
\emph{
Let $X$ be a valuative space. By $[X]_X$ we denote the subset of $X$ consisting of all its maximal points. If there is no ambiguity, we write $[X]$ for $[X]_X$. The \emph{separation map of $X$} denoted by $\text{sep}_X$
is defined to be the map from $X$ to $[X]$ which sends an element of $X$ to its unique maximal generization. Moreover, we endow $[X]$ with the quotient topology induced by $\sep_X$ and accordingly call $[X]$ the \emph{separated quotient of $X$}.
}
\end{defn}

\begin{rem}\label{rem separated quotient is T_1 quotient}
\emph{
Let $X$ be a valuative space. $[X]$ is a universal $T_1$-quotient of $X$ which means that $[X]$ is a $T_1$ space and that each continuous map from $X$ to a $T_1$-space factors uniquely through $[X]$.
}
\end{rem}
\begin{proof} First we show that $[X]$ is a $T_1$-space. Let $x\neq y \in [X]$. We do not distinguish between $x,y$ and their images under the separation map. Define $U:=X\setminus \overline{\lbrace x \rbrace }$ which is an open subset of $X$ containing $y$ but not $x$. Consider $\sep_X(U)\subseteq [X]$ which contains $y$. If it contained $x$ as well, we would find $u \in U$ such that $\sep_X(u)=x$. But this means $u \in \overline{ \lbrace x \rbrace}$ which is false. So $x \notin \sep_X(U)$. To show that $\sep_X(U)$ is an open subset of $[X]$, we have to show that $\sep_X^{-1}(\sep_X(U))$ is open in $X$. We claim $\sep_X^{-1}(\sep_X(U))=U$, where $"\supseteq"$ is clear. Let $z \notin U$, then $z \in \overline{\lbrace x \rbrace}$ and hence $\sep_X(z)=x \notin \sep_X(U)$. This proves the claim.\\
The rest of the proof is analogous to the respective part of the proof of \ref{prop. map: spa a to m(a)}:
Let $Y$ be a $T_1$-space and $f:X \rightarrow Y$ a continuous map. Let $x \in X$ and let $\tilde{x}$ be the maximal generization of $x$. We have to show that $f(x)$ equals $f(\tilde{x})$. Since $Y$ is a $T_1$-space, $f(\tilde{x})$ is closed in $Y$ and hence $f^{-1}(f(\tilde{x}))$ is a closed subset of $X$ that contains $\tilde{x}$. So $\overline{\lbrace \tilde{x} \rbrace} \subseteq f^{-1}(f(\tilde{x}))$ and therefore $f(x)=f(\tilde{x})$. \\
The continuity of the induced map $[X] \rightarrow Y$ is due to the quotient topology of $[X]$.
\end{proof}

\begin{rem}\label{rem: functoriality of sep map} (Functoriality of $\sep_X$) 
\emph{Let $f:X \rightarrow Y$ be a continuous map of valuative spaces. Then by the universal $T_1$-quotient property of $[X]$ there exists a unique continuous map $[f]:[X]\rightarrow [Y]$ such that the  diagram
\begin{gather*}
\begin{xy}
  \xymatrix{
      X \ar[r]^f \ar[d]_{\sep_X}    &   Y \ar[d]^{\sep_Y}  \\
     [X] \ar[r]_{[f]}           &   [Y]   
  }
\end{xy}
\end{gather*}
commutes.}
\end{rem}

It is natural that we want to consider those maps between valuative spaces that are not only continuous but also respect the additional structure of the respective separated quotients. This idea is made more precise in the following definition and remarks.

\begin{defn}\label{defn: valuative maps}
\emph{Let $f :X \rightarrow Y$ be a continuous map of valuative spaces. $f$ is called \emph{valuative} if $f([X]) \subseteq [Y]$.
}
\end{defn}

\begin{rem}\label{rem: valuative maps}\emph{
\begin{enumerate}
\item Let $X$ be a valuative space and $U\subseteq X$ an open subset. The immersion $U \hookrightarrow X$ is valuative since the maximal generizations of $U$ are those of $X$ contained in $U$.
\item If $f: X\rightarrow Y$ is a valuative map of valuative spaces, then $[f](x)=f(x)$ for all $x \in [X]$ (considered as a subset of $X$). 
\end{enumerate}
}
\end{rem}

\begin{rem}\label{rem: open subsets of valuative spaces}
\emph{Let $X$ be a valuative space and $U \subseteq X$ an open subset. Let $i: U \hookrightarrow X$ denote the canonical inclusion. In the commutative diagram
\begin{gather*}		
\begin{xy}
  \xymatrix{
      U \ar[r]^i \ar[d]_{\sep_U}    &   X \ar[d]^{\sep_X}  \\
     [U]_U \ar[r]_{[i]}           &   [X]_X   
  }
\end{xy}
\end{gather*}
$[i]$ is a continuous injection with image $[U]_X$ since the maximal generizations of $U$ are those of $X$ that are contained in $U$. In general this map may not induce a homeomorphism $[U]_{U} \cong [U]_X$. However, this statement holds in the following important special case:
}
\end{rem}

\begin{lem}\label{lem: case in which the map in fact is a homeomorphism}
Let $U$ be a coherent open subset of a valuative space $X$ such that $[X]$ is Hausdorff (we will see later that for example this is the case when $U$ is an affinoid open subset of a taut adic space $X$). Let $i: U \hookrightarrow X$ be the inclusion map. Then $[i]: \sep_U(U)\rightarrow \sep_X(X)$ is a closed topological embedding. In particular $\sep_U(U)$ is homeomorphic to $\sep_X(U)$.
\end{lem}
\begin{proof}
As explained above, we have a commutative diagram
\begin{gather*}
\begin{xy}
  \xymatrix{
      U \ar[r]^i \ar[d]_{\sep_U}    &   X \ar[d]^{\sep_X}  \\
     [U]_U \ar[r]_{[i]}           &   [X]_X   
  }
\end{xy}
\end{gather*}
where $i:U \rightarrow X$ denotes the inclusion. $[i]$ is a continuous injection, so it is enough to show that $[i]$ is closed. Let $V$ be a closed subset of $[U]_U$. Since $[U]_U$ is quasi-compact, $V$ is quasi-compact as well and hence also its image $[i](V)$. But $[X]_X$ is Hausdorff and therefore $[i](V)$ is also closed.
\end{proof}

\begin{exmp}\label{exmp: separation map adic spectra; adic spaces }
\emph{Let $A$ be a strictly $k$-affinoid algebra and let $X=\Spa(A,A^o)$ be the associated adic spectrum. The maximal points of $X$ are in bijection to the rank $1$ points of $X$ and hence we have a bijection between $\mt M(A)$ and $[X]$. $[X]$ and $\mt M(A)$ are actually homeomorphic since both spaces are maximal $T_1$-quotients of $X$. This shows that the map $q:X \rightarrow \mt M(A)$ can be considered as a special case of the separation map of the valuative space $X$ and $\mt M(A)$ can be regarded as its separated quotient.
}
\end{exmp}

\begin{exmp}\label{exmp: separated quotient affinoid adic space, affinoid Berkovich space}
\emph{
Let $A$ and $B$ be strictly $k$-affinoid algebras and $(\varphi,\varphi^b)$ be a morphism of affinoid adic spaces from $\Spa(A,A^o)$ to $\Spa(B,B^o)$. Then $\varphi$ is induced by a homomorphism of Huber pairs $f: (B,B^o) \rightarrow (A,A^o)$ (cf. \ref{theo. fully faithful functor}) and we have a  commutative diagram
\begin{gather*}		
\begin{xy}
  \xymatrix{
      \Spa(A,A^o) \ar[r]^\varphi \ar[d]_{\sep_{\Spa(A,A^o)}}    &   \Spa(B,B^o) \ar[d]^{\sep_{\Spa(B,B^o)}}  \\
     \mt M(A) \ar[r]_{[\varphi]}           &   \mt M(B) .  
  }
\end{xy}
\end{gather*}
Since $\varphi$ is valuative by \ref{prop: properties of analytic adic spaces}, $[\varphi](x)$ equals $\varphi(x)=f\circ x$ for all $x \in \mt M(A)$. Hence $[\varphi]$ is precisely the morphism of $k$-affinoid spaces induced by the morphism of $k$-affinoid algebras $\tilde{f}:B\rightarrow A$, which corresponds to $f$ via the equivalence in \ref{prop. equivalence of cat. k-affinoid algebras, complete Huber pairs locally of finite type /k}.
}
\end{exmp}

\subsection{Reflexive valuative spaces}
Let $A$ be a strictly $k$-affinoid algebra and $X=\Spa(A,A^o)$ be the corresponding affinoid adic space. We have already seen that all information of $X$ is contained in $A$ and hence we do not lose any information if we pass to its separated quotient $\mt M(A)$. One main question is under which conditions the topological aspects of this result can be retained true in a more general setting where $X$ is a certain general adic space. Considering a general valuative space we can search for properties that ensure that sufficient much structure can be retained by its separated quotient. 
The following definition is central for this idea:

\begin{defn}\label{defn: reflexive valuative spaces}
\emph{
Let $X$ be a valuative space. $X$ is said to be \emph{reflexive} if for two coherent open subsets $U\subseteq V$ of $X$, such that $[U]=[V]$, we have $U=V$. 
}
\end{defn}

Reflexiveness gives us a hint how we can interpret set-theoretical assertions of a separated quotient in terms of the underlying valuative space.

\begin{prop}\label{prop: reflexive spaces alternative description}
Let $X$ be a valuative space. Then $X$ is reflexive if and only if for any pair of open subsets $U \subseteq V $ of $X$, such that the inclusion $U\hookrightarrow V$ is quasi-compact, $[U]=[V]$ implies $U=V$.
\end{prop}
\begin{proof}
\cite{FK} Proposition 0.2.4.3.
\end{proof}

Note that the above proposition implies that \emph{reflexiveness} is a local property:

\begin{lem}\label{lemma: reflex, local property}
Let $X$ be a valuative space such that there exists a covering $(W_i)_{i \in I}$ of reflexive quasi-compact open subsets of $X$. Then $X$ is reflexive.
\end{lem}
\begin{proof}
Let $U \subseteq V$ be open subsets of $X$ such that $U\hookrightarrow V$ is quasi-compact and $[U]=[V]$. Since the $W_i$'s are quasi-compact and $X$ is quasi-separated, the inclusions $U\cap W_i \hookrightarrow V \cap W_i$ are quasi-compact. Since $U$ and $V$ are open we also have $[W_i \cap U]_{W_i}=[W_i \cap V]_{W_i}$ and hence $(W_i \cap V) = (W_i \cap U)$ by the reflexiveness of $W_i$. But this means $U=V$.
\end{proof}

\begin{lem}\label{lem: closure of retrocompact opens}
Let $X$ be a valuative space and $U$ a retrocompact open subset of $X$. Then we have $\overline{U}=\bigcup_{x \in U} \overline{ \lbrace x \rbrace}=sep^{-1}_X([U])$.
\end{lem}
\begin{proof}
\cite{FK} Corollary 0.2.2.27, the second equality holds by definition.
\end{proof}

\begin{prop}\label{prop: reflexive means regular for the right sets and U=U'}
Let $X$ be a valuative space. Then any retrocompact open subset of $U\subseteq X$ is \emph{regular}, that is $(\overline{U})^{\circ}=U$. Moreover, whenever we have $[U]=[U']$ for retrocompact open subsets $U$ and $U'$ of $X$, then $U$ and $U'$ coincide. 
\end{prop}
\begin{proof}
The first part is \cite{FK} Proposition 0.2.4.5. For the second part assume $U,U'$ as above with $[U]=[U']$. Then we have $U=(\overline{U})^o =(\sep_X^{-1}([U]))^o=(\sep_X^{-1}([U']))^o=(\overline{U'})^o=U'$ where the second and the forth equality are due to \ref{lem: closure of retrocompact opens}. 
\end{proof}

\begin{prop}\label{prop: maps of reflexive spaces}
Let $f,g:X \rightarrow Y$ be two valuative quasi-compact maps of valuative spaces. Assume that $X$ is reflexive and $[f]=[g]$. Then we have $f=g$.
\end{prop}
\begin{proof}
\cite{FK} Corollary 0.2.4.12.
\end{proof}

Since reflexiveness is a local property, the following proposition shows that an adic space, that is locally of finite type over $k$, is reflexive. This fact is one key point of all constructions following in §7.\\  
 Keep in mind that our idea is to consider an adic space as a valuative space and define the structure of a $k$-analytic Berkovich space on its separated quotient. The result of the lemma before brings us in the situation to deduce properties of maps between adic spaces from maps between the corresponding Berkovich spaces.

\begin{prop}\label{prop. affine adic spaces connected to strict k-aff alg are reflexive}
Let $A$ be a strictly $k$-affinoid algebra. Then $\Spa(A,A^o)$ is reflexive as a topological space. In particular, the topological space underlying any adic space, that is locally of finite type over $k$, is reflexive.
\end{prop}
\begin{proof}
\cite{FK} Chapter II Proposition C.2.6.
\end{proof}

 The next result will play an important role for the construction of a functor from finite adic to strict Berkovich spaces.
 
\begin{lem}\label{lemma: connection inclusion of subsets on separated quotient}
Let $f:X \rightarrow Y$ be a locally quasi-compact valuative map of valuative spaces where $X$ is reflexive. Let $U \subseteq Y$ be an open coherent subset such that $[f]([X]) \subseteq [U]$. Then we already have $f(X) \subseteq U$.
\end{lem}
\begin{proof}
Since $U$ is coherent and $Y$ is quasi-separated, $U$ is retrocompact and hence $f^{-1}(U)$ is a retrocompact open subset of $X$ by \ref{prop: quasi-compact map between locally coherent spaces}. But this means that the inclusion $f^{-1}(U) \hookrightarrow X$ is quasi-compact. Moreover, we have $[f^{-1}(U)]=f^{-1}(U)\cap [X] = \lbrace x \in [X] \ | \ f(x) \in U \rbrace = \lbrace x \in [X] \ | \ [f](x)\in [U] \rbrace = [f]^{-1}([U])=[X]$.
Hence we get $f^{-1}(U)=X$ by \ref{prop: reflexive spaces alternative description}, i.e. $f(X) \subseteq U$.
\end{proof}

In the situation of the lemma above, if $V$ is a coherent subset of $X$ such that $[f]([V]) \subseteq [U]$, then we also get $f(V)\subseteq U$ since we can apply the lemma to $f|_V:V\rightarrow Y$.\\

In the setting of $k$-affinoid Berkovich spaces the global objects are obtained by glueing $k$-affinoid spaces in a certain way along affinoid domain embeddings. Hence we have to figure out which kind of map between affinoid adic spaces we obtain from such maps by the equivalence in §4.1. It will turn out, that we get open immersions of adic spaces and hence we will be able to transfer the glueing process back and forth. Our strategy to obtain the mentioned result is to consider rational subsets and then use the Gerritzen-Grauert theorem \ref{theo. gerritzen Grauert} to receive the general result. 

The following lemma can also be deduced by comparing the universal properties of the rings involved. However, with the notions we introduced we can carry it to a geometric assertion and use the universal properties of the associated geometric spaces.

\begin{lem}\label{lem. same universal property for rational subsets}
Let $A$ be a strictly $k$-affinoid algebra and $f_1,...,f_n, g \in A $ such that the ideal generated by such elements is $A$. Let $f:= \lbrace f_1,...,f_n \rbrace$. Then we have an isomorphism of $k$-Banach algebras 
\begin{gather*}		
A \langle \frac{f}{g}\rangle \cong A\lbrace T_1,...,T_n \rbrace / (gT_i-f_i)_{i=1,...,n}.
\end{gather*}
\end{lem}
\begin{proof}
Let 
\begin{gather*}		
U:= \lbrace x \in \Spa(A,A^o) \ | \ \val{f_i(x)} \leq \val{g(x)} \ \forall \ i=1,...,n\rbrace   
\end{gather*} and
\begin{gather*}		
V:= \lbrace x \in \mt M(A) \ | \ \val{f_i(x)} \leq \val{g(x)} \ \forall \ i=1,...,n\rbrace.
\end{gather*}
 Identifying the separated quotient of $\Spa(A,A^o)$ with $\mt M(A)$ (cf. \ref{exmp: separation map adic spectra; adic spaces }) we have $[U]=V$. 
It is enough to show that $l:\mt M(A \langle \frac{f}{g}\rangle) \rightarrow \mt M(A)$ is an affinoid domain embedding with image $V$, since this implies the claim by the universal property of the affinoid domain embedding. It is clear that $\im(l) \subseteq V$. So let $B$ be a strictly $k$-affinoid algebra and $\varphi: \mt M(B) \rightarrow \mt M(A)$ a morphism of strictly $k$-affinoid spaces with image contained in $V$. This induces a morphism $\psi:\Spa(B,B^o)\rightarrow \Spa(A,A^o)$ of adic space with $\im(\psi) \subseteq U$ by \ref{lemma: connection inclusion of subsets on separated quotient}. The universal property of $\Spa(A\langle \frac{f}{g}\rangle,A\langle \frac{f}{g}\rangle^o) \rightarrow \Spa(A,A^o)$ gives us a unique continuous ring homomorphism $A\langle \frac{f}{g}\rangle \rightarrow B$ that commutes with the respective further morphisms (cf. \ref{lem. embeddings of rational subsets}). This in turn induces a unique morphism $h: \mt M(B) \rightarrow \mt M(A\langle \frac{f}{g}\rangle)$ of $k$-affinoid spaces such that $\varphi=h \circ l$.
\end{proof}

We can use the result above to compare the open immersions of adic spaces and affinoid domain embeddings both induced by rational subsets in the respective setting.

\begin{rem}\label{rem. rational subsets and the separation map}
\emph{
Let $A$ be a strictly $k$-affinoid algebra, $X=\Spa(A,A^o)$ and $X(\frac{f}{s})$ (where $f= \lbrace f_1,...,f_n \rbrace$) a rational subset of $A$. Since $A\langle \frac{f}{s}\rangle$ and $ A\lbrace T_1,...,T_n \rbrace / (sT_i-f_i)_{i=1,...,n}$ are isomorphic by \ref{lem. same universal property for rational subsets}, we get a commutative diagram
\begin{gather*}		
\begin{xy}
  \xymatrix{
      \Spa(A\langle \frac{f}{s}\rangle,A\langle \frac{f}{s}\rangle^o) \ar[r] \ar[d]_{sep}    &   \Spa(A,A^o) \ar[d]^{sep}  \\
     \mt M(A\lbrace T_1,...,T_n \rbrace / (sT_i-f_i)) \ar[r]           &   \mt M(A)   .
  }
\end{xy}
\end{gather*}
The upper horizontal arrow is an open immersion of adic spaces (i.e. it is an open topological embedding and induces an isomorphism of adic spaces onto its image) and the lower horizontal arrow is an affinoid domain embedding of $k$-affinoid spaces.
}
\end{rem}

So if we start with an affinoid domain embedding inducing a rational domain on the affinoid Berkovich side, we get an open immersion of adic spaces between the corresponding adic spectra. The following result will generalize this observation:

\begin{prop}\label{prop. affinoid domain embedding open immersion of adic spaces}
Let $U$ be an affinoid domain in the strictly $k$-affinoid space $\mt M(A)$. Let $f:Y:=\Spa(A_U,A_U^o)\rightarrow \Spa(A,A^o):=X$ denote the corresponding morphism of adic spaces. Then $f$ is an open immersion of adic spaces, i.e. it is an open topological embedding and induces an isomorphism $(Y, \mt O_Y, (v_y)_{y \in Y}) \cong (f(Y), \mt O_{X|_{f(Y)}} , (v_x)_{x \in f(Y)})$ of adic spaces.
\end{prop}
\begin{proof}
Let $\psi : \mt M(A_U) \rightarrow \mt M(A)$ be the obtained affinoid domain embedding.
We first show that $f$ is an open topological immersion. Since $f$ is injective by \cite{FK} Proposition C.2.8. (note that affinoid domains in \cite{FK} satisfy by definition the stronger universal property; cf. \ref{prop. universal property holds for all k- banach algebras}), it suffices to show that $f$ is open:\\
By the Gerritzen-Grauert theorem \ref{theo. gerritzen Grauert} we can find rational domains $U_1,...,U_n$ in $\mt M(A)$ such that $U=\bigcup_{i=1}^n U_i$. \ref{rem. pre images of affinoid domains} shows that $V_i:=\psi^{-1}(U_i)$ is a rational domain of $\mt M(A_U)$  and that $A_{U_i} \cong (A_U)_{V_i}$ for each $i=1,...,n$. 
We denote the respective affinoid domain embeddings $\mt M(A_{U_i}) \rightarrow \mt M(A_U)$ by $\psi_i$ and the affinoid domain embeddings $\mt M(A_{U_i})\rightarrow \mt M(A)$ by $\varphi_i$. Note that those maps represent rational domains and commute with $\psi$ (i.e. $\varphi_i=\psi_i \circ \psi$ for all $i=1,...,n$).  Hence for $i=1,...,n$ the corresponding morphisms of affinoid adic spaces 
\begin{gather*}		
g_i:\Spa(A_{U_i},A_{U_i}^o)\rightarrow \Spa(A_U,A_U^o) \ \text{and} \ h_i: \Spa(A_{U_i},A_{U_i}^o)\rightarrow \Spa(A,A^o)
\end{gather*}
 are open immersions of adic spaces (cf. \ref{exmp: separated quotient affinoid adic space, affinoid Berkovich space} ii)) that satisfy $h_i =f \circ g_i$.\\
We have $[\bigcup_{i=1}^n \im(g_i)] = [\Spa(A_U,A_U^o)]$ since $U=\bigcup_{i=1}^n U_i$. So $\bigcup_{i=1}^n \im(g_i)$ is a retrocompact open subset of the reflexive valuative space $\Spa(A_U,A_U^o)$ with full separated quotient. By definition of reflexiveness this includes that $\bigcup_{i=1}^n \im(g_i) = \Spa(A_U,A_U^o)$. Hence the open subsets $W$ of $\Spa(A_U,A_U^o)$, such that $W$ is contained in some $g_i(\Spa(A_{U_i},A_{U_i}^o))$, form a basis of the topology of $\Spa(A_U,A_U^o)$. But the image of such subsets under $f$ is open since $h_i$ is open for each $i=1,...,n$. Therefore $f$ is an open topological immersion.\\
 Now let $W'$ be a rational of $X$ that is contained in the image of $f$. This means
 \begin{gather*}		
 [W'] \subseteq [f(\Spa(A_U,A_U^o))]=\psi([\Spa(A_U,A_U^o)])=\psi( \mt M(A_U))=U.
  \end{gather*} 
  In remark \ref{rem. pre images of affinoid domains} we have seen that in this case $A_{[W']} \cong (A_U)_{\psi^{-1}([W'])}$. Therefore $\mt O_X (W') \cong \mt O_Y(f^{-1}(W'))$ by \ref{exmp: separated quotient affinoid adic space, affinoid Berkovich space}.  Since sets of this type form a basis of the topology of $\im(f)$, their pre-images under $f$ are rational subsets of $Y$ that form a basis of the topology of $Y$. Hence the respective sheaves are isomorphic on a basis of the topology which implies the final claim.
\end{proof}

\subsection{Locally strongly compact valuative spaces}
Let $X$ be a valuative space. In general the separated quotient of $X$ is just a $T_1$-space but not locally Hausdorff. Since the topological space underlying a Berkovich analytic space has to have this property we need to find a sufficient condition on $X$ that ensures $[X]$ to be locally Hausdorff.  One such particular property will be introduced in this subsection.

\begin{defn}\label{def. locally strongly compact valuative spaces }
\emph{
A valuative space $X$ is said to be \emph{locally strongly compact} if for any $x \in X$ there exists a pair $(W_x,V_x)$ consisting of an open subset $W_x$ that is stable under specialization (i.e. for any $y \in W_x$ we have $\overline{\lbrace y \rbrace}\subseteq W_x$) and a coherent open subset $V_x$ such that $x \in W_x \subseteq V_x$.  
}
\end{defn}

\begin{prop}\label{prop: locally strongly compact = taut}
For a quasi-separated valuative space $X$ the following conditions are equivalent:
\begin{enumerate}
\item $X$ is locally strongly compact;
\item for any quasi-compact open subset $U \subseteq X$, the topological closure $\overline{U}$ is still quasi-compact in $X$.
\end{enumerate}
\end{prop}

\begin{proof}
\cite{FK} Proposition 0.2.5.5.
\end{proof}

\begin{rem}\label{rem: taut= locally strongly compact}
\emph{
In particular the proposition above states that a valuative space is taut if and only if it is quasi-separated and locally strongly compact.
} 
\end{rem}

\begin{theo}\label{theo: localy strongly compact -> locally Hausdorff}
Let $X$ be a locally strongly compact valuative space. Then the separated quotient $[X]$ is locally compact and in particular locally Hausdorff (for the definitions cf. \ref{def. sober quasi sepatated...}). If in addition $X$ is quasi-separated, then $[X]$ is Hausdorff.
\end{theo}
\begin{proof}
\cite{FK} Theorem 0.2.5.7. and Corollary 0.2.5.9.
\end{proof}

Moreover, in a locally strongly compact valuative space we have a criterion to find out when the images under the separation map of a certain family of subsets of $X$ form a net on the separated quotient $[X]$. This will be a useful tool for our constructions in §7. 

\begin{prop}\label{prop: when net on separated quotient}
Let $X$ be a locally strongly compact valuative space and $(U_\alpha)_{\alpha \in L}$ a covering of $X$ consisting of quasi-compact open subsets. Then the images under the separation map $([U_\alpha])_{\alpha \in L}$ form a quasi-net on $[X]$. For this quasi-net to be a net it is necessary and sufficient that for each $\alpha,\beta \in L$ the sets in $\lbrace U_\gamma \ | \ U_\gamma \subseteq U_\alpha \cap U_\beta \rbrace$ form a covering of $U_\alpha \cap U_\beta$.
\end{prop}
\begin{proof}
\cite{FK} Proposition 0.2.6.2.
\end{proof}

\section{Valuations on topological spaces}
In the previous sections we assigned the separated quotient to a valuative space. Following $[FK]$ we now approach the question how this process can be inverted. At first we will associate a coherent reflexive valuative space to a certain compact space and then generalize this construction to locally Hausdorff spaces that are endowed with some additional structure. Moreover, we will explain how those notions are connected to the interplay of adic and Berkovich spaces.

\begin{defn}\label{def. valuations on compact sets}
\emph{
Let $S$ be a compact topological space.
\begin{enumerate}
\item A distributive sublattice $v$ of $\textbf{2}^S$ is called \emph{valuation of $S$} if there exists a coherent reflexive valuative space $X_S$ and a continuous $\pi :X_S \rightarrow S$ such that:
	\begin{enumerate}
	\item the map $[X_S]\rightarrow S$ induced from $\pi$ is a homeomorphism;
	\item the lattice $v$ coincides with the lattice $\lbrace [U] \subseteq S \ | \ U \in \QCOuv(X_S) \rbrace$ by $[X_S]\cong S$ (here $\QCOuv(X_S)$ denotes the set of all open and quasi-compact subsets of $X_S$; cf. \ref{def. sober quasi sepatated...} vii)).
	\end{enumerate}
\item If $v$ is a valuation of $S$, then the pair $(S,v)$ is called a \emph{valued compact space}.
\end{enumerate}
}
\end{defn}

\begin{rem}\label{rem. canonical representation of X_s... Stone}
\emph{Let $(S,v)$ be a valued compact space. The coherent reflexive valuative space $X_S$ from the definition above is uniquely determined up to unique homeomorphism and homeomorphic to $\Spec \ v$ where $\Spec \ v$ is the topological space corresponding to the lattice $v$ via \emph{Stone duality} (cf. \cite{FK} Theorem 0.2.2.8). However, for this paper we will only consider the case in which $S= \mt M(A)$ is a strictly $k$-affinoid space endowed with a valuation induced by the separation map $\Spa(A,A^o) \rightarrow \mt M(A)$ (as will be explained in the following example). In this situation we can assume that the coherent reflexive valuative space $X_S$ simply is equal to $\Spa(A,A^o)$.
}
\end{rem}

\begin{exmp}\label{exmp. valuation in M(A) }
\emph{
Let $A$ be a strictly $k$-affinoid algebra. Then $X:=\Spa(A,A^o)$ is a reflexive valuative space by \ref{prop. affine adic spaces connected to strict k-aff alg are reflexive} and as seen in \ref{exmp: separation map adic spectra; adic spaces } we have a homeomorphism $[\Spa(A,A^o)]\rightarrow \mt M(A)$. Hence we obtain a valued compact space $(\mt M(A),v)$, where we define $v:= \lbrace [U]  \subseteq \mt M(A) \ | \ U \in \QCOuv(\Spa(A,A^o))\rbrace$. Note that for this construction it is necessary that $X$ is reflexive. So for a possible generalization to arbitrary affinoid adic spaces this is a huge obstacle.
}
\end{exmp}

The previous example shows how we mainly want to use the notion of compact valued spaces. We have also obtained that it fits well with our local setting. To generalize this idea to a connection between  non-affinoid adic and Berkovich spaces, we need the following generalization:

\begin{defn}\label{def. valuations of locally Hausdorff spaces}
\emph{
Let $X$ be a locally Hausdorff space.
\begin{enumerate}
\item A \emph{pre-valuation on $X$} denoted by $v=(\tau(v), \lbrace v_S \rbrace_{S \in \tau(v)})$ consists of:
\begin{enumerate}
	\item a net $\tau(v)$ of compact subsets of $X$,
	\item for each $S \in \tau(v)$ a valuation $v_S$
\end{enumerate}
such that the following condition holds:\\
for $S,S' \in \tau(v)$ with $S \subseteq S'$, we have $v_S=\lbrace T \in v_{S'} \ | \ T \subseteq S \rbrace$.
\item A pre-valuation $v=(\tau(v), \lbrace v_S \rbrace_{S \in \tau(v)})$ on $X$ is said to be a \emph{valuation on $X$} if $\bigcup_{S \in \tau(v)} v_S = \tau(v)$ and for any collection of finitely many elements $S_1,...,S_n \in \tau(v)$ their union $S=\bigcup_{i=1}^n S_i$ belongs to $\tau(v)$ whenever $S$ is Hausdorff as a subspace of $X$.
\item If $v$ is a valuation on $X$, then the pair $(X,v)$ is said to be a \emph{valued locally Hausdorff space}.
\end{enumerate}
}
\end{defn}

\begin{exmp}\label{exmp. Berkovich spaces as valued locally Hausdorff spaces }
\emph{
Let $X=(X,\mt A, \tau)$ be a strictly $k$-analytic Berkovich space. We want to endow $X$ with the structure of a valued locally Hausdorff space. Hence we have to define a valuation $v_S$ on $S$ for each $S \in \tau$. That can be done analogously to \ref{exmp. valuation in M(A) }, namely by the separation map $\Spa(A_S,A_S^o) \rightarrow \mt M(A_S)\cong S$ (where $A_S$ denotes $\mt A(S)$). So $X_S=\Spa(A,A^o)$ is an associated valuative space of $(S,v_S)$. For $S,S' \in \tau$ with $S\subseteq S'$ we have an affinoid domain embedding $\mt M(A_S) \rightarrow \mt M(A_{S'})$ and hence the induced morphism of affinoid adic spaces $\Spa(A_S,A_S^o) \rightarrow \Spa(A_{S'},A^o_{S'})$ is an open embedding of topological spaces by \ref{prop. affinoid domain embedding open immersion of adic spaces}). Therefore 
\begin{gather*}		
\QCOuv(\Spa(A_S,A_S^o))=\QCOuv(\Spa(A_{S'},A^o_{S'}))|_{(\Spa(A_S,A_S^o))}
\end{gather*}
 and hence $v=(\tau, v_S)$ gives a pre-valuation on $X$.
}
\end{exmp}

\begin{theo}\label{theo. glueing valued spaces}
Let $X$ be a locally Hausdorff space and $v=(\tau(v), \lbrace v_S \rbrace_{S \in \tau(v)})$ be a pre-valuation on $X$. For $S \in \tau(v)$ let $\Spec\  v_S$ denote the reflexive coherent valuative space associated to $v_S$ via Stone duality (cf. \ref{rem. canonical representation of X_s... Stone}). We define 
\begin{gather*}		
\Spec \ v:= \lim\limits_{\underset{S \in  \tau(v)}\longrightarrow} \Spec\  v_S.
\end{gather*}
Then $\Spec \ v$ is a reflexive valuative space and its separated quotient $[\Spec \ v]$ is homeomorphic to $X$. 
\end{theo}
\begin{proof}
\cite{FK} Theorem 0.2.6.16.
\end{proof}

\begin{rem}\label{rem. on glueing valued spaces}
\emph{
Applying this result to the situation of the above example \ref{exmp. Berkovich spaces as valued locally Hausdorff spaces } ($X=(X, \mt A , \tau)$ is a strictly $k$-analytic Berkovich space) we get the following statement (since in this situation $\Spec \ v_S$ is homeomorphic to $\Spa(A_S,A_S^o)$):
\begin{gather*}		
X^{ad}:=  \lim\limits_{\underset{S \in  \tau(v)}\longrightarrow} \Spa(A_S,A_S^o)
\end{gather*}
is a reflexive valuative space (homeomorphic to $\Spec \ v$), with $[X^{ad}] \cong X$. There is also an explicit description of $X^{ad}$, namely
\begin{gather*}		
(\bigsqcup_{S \in \tau} \Spa(A_S,A_S^o))/ \sim.
\end{gather*}
Here the equivalence relation $\sim$ is generated by the relation $\sim_R$ which is defined as follows:
For $x \in \Spa(A_S,A_S^o)$ and $y \in \Spa(A_{S'},A_{S'}^o)$ we have $x \sim_R y$ if there exists $S'' \in \tau$ such that $S,S'\subseteq S''$ and the induced morphisms $\Spa(A_S,A_S^o) \rightarrow \Spa(A_{S''},A_{S''}^o)$ and $\Spa(A_{S'},A_{S'}^o) \rightarrow \Spa(A_{S''},A_{S''}^o)$ identify $x$ and $y$ in $\Spa(A_{S''},A_{S''}^o)$.\\
Note that by this description and \ref{prop. affinoid domain embedding open immersion of adic spaces} one sees that the canonical maps $\pi_S:\Spa(A_S,A^o_S)\rightarrow X^{ad}$ are open topological embeddings. We will later use this construction to define the structure of an adic space on $X^{ad}$ in §7.
}
\end{rem}

Now we investigate some more properties of $\Spec \ v$ and transfer them back to $X^{ad}$.

\begin{prop}\label{prop. extension of valuations}
Any pre-valuation $v$ of a locally Hausdorff space $X$ has a unique extension to a valuation.
\end{prop}
\begin{proof}
\cite{FK}  Corollary 0.2.6.17.
\end{proof}

\begin{prop}\label{prop. properties of glued valued space}
Let $(X,v)$ be a locally Hausdorff valuative space.
The valuative space $\Spec \ v$ is locally strongly compact. If $X$ is Hausdorff, then $\Spec \ v$ is quasi-separated.
\end{prop}
\begin{proof}
\cite{FK} Proposition 0.2.6.18
\end{proof}

\begin{rem}\label{rem. properties of valued spaces for Berkovich}
\emph{
As in the examples above (cf. \ref{exmp. Berkovich spaces as valued locally Hausdorff spaces }) let $X=(X,\mt A, \tau)$ be a strictly $k$-analytic Berkovich space. By \ref{prop. extension of valuations} we have a unique extension of the pre-valuation on $X$ defined in \ref{exmp. Berkovich spaces as valued locally Hausdorff spaces } to a valuation on $X$. Therefore $X^{ad}$ from \ref{rem. on glueing valued spaces} is locally strongly compact as seen in  \ref{prop. properties of glued valued space}. Moreover, if $X$ is Hausdorff, then $X^{ad}$ is taut. This provides the topological aspects of the main theorem of this paper. In the following final section we will focus on the comparison of the further structure.
}
\end{rem}

\section{Final result -  an equivalence of categories}

We have collected sufficiently many properties of valuative spaces and associated separated quotients to be able to define a functor from the category of taut adic spaces that are locally of finite type over $k$ to the category of strictly Hausdorff $k$-analytic Berkovich spaces. The first construction describes the functor on the object level:

\begin{const}\label{construction: functor objects}
\emph{Let $(X, \mt O_X , (v_x)_{x \in X})$ be a taut adic space that is locally of finite type over $k$. We want to define the structure of a strictly Hausdorff $k$-analytic Berkovich space on $[X]$ which denotes the separated quotient of $X$. For this purpose we successively process definition \ref{def. k-analytic berkovich spaces}.\\
 In view of \ref{prop: locally strongly compact = taut} we know that $X$ is quasi-separated and locally strongly compact and hence $[X]$ is \emph{Hausdorff} by \ref{theo: localy strongly compact -> locally Hausdorff}.\\
The next step is to define a suitable net on $[X]$ on which we can define a $k$-affinoid atlas afterwards. Let $\delta := \lbrace U \subseteq X \ | \ U \ \text{is an affinoid open subset} \rbrace$. This means that $U\cong \Spa(\mt O_X(U),\mt O_X(U)^o)$ where $(\mt O_X(U),\mt O_X(U)^o)$ is a Huber pair which is topologically of finite type over $k$ by \ref{rem. locally of finite type over k morphisms}. \ref{prop. equivalence of cat. k-affinoid algebras, complete Huber pairs locally of finite type /k} states that in this case $\mt O_X(U)$ is a strictly $k$-affinoid algebra. For brevity, we put $A_U:= \mt O_X(U)$. Since $\delta$ forms a basis of the topology of $X$, its image under the separation map $\tau:=\text{sep}_X(\delta) = \lbrace [U] \ | \ U\in \delta \rbrace$ then is a net on $[X]$ by \ref{prop: when net on separated quotient}.\\
We set $\mt A([U]):= A_U$ to endow $[X]$ with a $k$-affinoid atlas $\mt A$ with net $\tau$. This is well defined since $X$ is reflexive. Indeed for $U,V \in \delta$ with $[U]=[V]$ we already have $U=V$ by \ref{prop: reflexive means regular for the right sets and U=U'}. For each $U \in \delta$ we know from \ref{lem: case in which the map in fact is a homeomorphism} that $\sep_X(U)$ is homeomorphic to $\sep_U(U)$ which in turn is homeomorphic to $\mt M(A_U)$ by \ref{exmp: separation map adic spectra; adic spaces }. \\
Now let $[U],[U']\in \tau$ such that $[U]\subseteq [U']$. Again reflexiveness of $X$ provides $U \subseteq U'$ by the first assertion of \ref{prop: reflexive means regular for the right sets and U=U'}. Therefore the restriction map $\text{res}_{U'}^{U}$ of $\mt O_X$ provides a morphism $\rho_{U'}^U:=\text{res}_{U'}^{U}:A_{U'}= \mt O_X (U')\rightarrow \mt O_X(U)= A_U$ of strictly $k$-affinoid algebras.\\ Next we have to show that the associated morphism of $k$-affinoid spaces $\psi:\mt M(A_U)\rightarrow \mt M(A_{U'})$ is an affinoid domain embedding. We have a commutative diagram
\begin{gather*}		
\begin{xy}
  \xymatrix{
      \Spa(A_U,A_U^o) \ar[r]^\varphi \ar[d]_{\sep_{\Spa(A_{U},A_{U}^o)}}    &   \Spa(A_{U'},A_{U'}^o) \ar[d]^{\sep_{\Spa(A_{U'}, A_{U'}^o)}}  \\
     \mt M(A_U) \ar[r]_{\psi}           &   \mt M(A_{U'})   
  }
\end{xy}
\end{gather*}
in which the vertical arrows are the separation maps and the horizontal arrows are the respective maps induced by $\rho_{U'}^U$.
The explicit description of the bijection in \ref{theo. fully faithful functor} shows that $\varphi$ corresponds to the open immersion of adic spaces $U\hookrightarrow U'$ and hence the image of $\psi$ is precisely $[U]$. Now we consider an arbitrary morphism of strictly $k$-affinoid spaces $\pi :Y= \mt M(B) \rightarrow \mt M(A_{U'})$, such that $\im(\pi) \subseteq [U]$. Since $\pi$ comes from a morphism of strictly $k$-affinoid algebras, it induces a morphism of adic spectra and we obtain a commutative diagram
\begin{gather*}		
\begin{xy}
  \xymatrix{
      \Spa(B,B^o) \ar[r]^f \ar[d]_{\sep_{\Spa(B,B^o)}}    &   \Spa(A_{U'},A_{U'}^o) \ar[d]^{\sep_{\Spa(A_{U'}, A_{U'}^o)}}  \\
     \mt M(B) \ar[r]_{\pi}           &   \mt M(A_{U'})   .
  }
\end{xy}
\end{gather*}
By \ref{lemma: connection inclusion of subsets on separated quotient} we get $\im(f)\subseteq U$ and hence there exists a unique morphism of affinoid adic spaces $g:\Spa(B,B^o)\rightarrow \Spa(A_U,A_U^o)$ such that $f=\varphi \circ g$. But this induces a unique morphism $\mt M(B) \rightarrow \mt M(A_U)$ of strictly $k$-affinoid spaces such that the diagram
\begin{gather*}		
\begin{xy}
  \xymatrix{
      \mt M(A_U) \ar[rr]^\psi   &     &  \mt M(A_{U'})   \\
                             &  \mt M(A)  \ar[ru]_\pi \ar[ul]^h &
  }
\end{xy}
\end{gather*}
commutes.
So $[U]$ is an affinoid domain of $\mt M(A_{U'})$. For $U,U',U'' \in \delta$ such that $[U] \subseteq [U'] \subseteq [U'']$ we have $\rho_U^{U''}=\rho_U^{U'}\circ \rho_{U'}^{U''}$ by the cocycle property of $\mt O_X$. Hence the triple $([X],\mt A, \tau)$ is a Hausdorff strictly $k$-analytic Berkovich space.
}
\end{const}

The next paragraph approaches morphisms between adic spaces that are locally of finite type over $k$. It shows how the construction from \ref{construction: functor objects} can be used to define morphisms between the associated Hausdorff strictly $k$-analytic Berkovich spaces.

\begin{const}\label{const: functor morphisms}
\emph{
Let $(f,f^b):(X,\mt O_X, (v_x)_{x \in X}) \rightarrow (Y,\mt O_Y, (v_y)_{y \in Y})$ be a morphism of taut adic spaces that are locally of finite type over $k$. We endow $[X]=([X],\mt A_X,\tau_X)$ and $[Y]=([Y], \mt A_Y, \tau_Y)$ with the structure of a $k$-affinoid Berkovich space as we did in the construction above (\ref{construction: functor objects}). We write $A_U$ for $\mt A_X([U])$ and $A'_V$ for $\mt A_Y([V])$ where $[U]$ in $\tau_X$ and $[V]$ in $\tau_Y$. Since $f$ is a continuous map between valuative spaces, it induces a continuous map $[f]:[X]\rightarrow [Y]$ which commutes with the respective separation maps (cf. \ref{rem: functoriality of sep map}).\\
We want $[f]$ to induce a strong morphism from some Hausdorff strictly $k$-analytic Berkovich space $Z=(Z,\mt A_Z, \tau_Z)$ to $([Y],\mt A_Y,\tau_Y)$ where $(\mt A_Z, \tau_Z)$ is an appropriate refinement of $(\mt A_{X}, \tau_{X})$ (more precisely, we want to have a quasi-isomorphism $Z \rightarrow [X]$). So we have to ensure that for each $[U]\in \tau_{Z}$ there is an $[U'] \in \tau_Y$ such that $f([U])\subseteq [U']$. For this purpose we coarsen the $k$-affinoid atlas of $X$ in the following way:
Let $\delta_X$ and $\delta_Y$ be defined as in the previous construction (namely the sets containing all open affinoid subsets of $X$, respectively $Y$) and set
\begin{gather*}		
\delta_X|_{f}:=\lbrace V \in \delta_X \ | \ \exists \ V' \in \delta_Y \ \text{with} \ V \subseteq f^{-1}(V') \rbrace. 
\end{gather*}
Since $(f^{-1}(V'))_{(V'\in \delta_Y)}$ is an open covering of $X$ the sets of $\delta_X|_{f}$ form a basis of the topology of $X$. Therefore $\tau_X|_f:=\sep_X(\delta_X|_{f}) \subseteq \tau_X$ is a net on $[X]$ by \ref{prop: when net on separated quotient}. We denote the restriction of the $k$-affinoid atlas $\mt A_X$ to $\tau_X|_f$ by $\mt A_X|_f$. Then the identity $\id_{[X]}:([X], \mt A_X|_f, \tau_X|_f) \rightarrow ([X],\mt A_X, \tau_X)$ is a quasi-isomorphism of $k$-affinoid spaces.\\
We claim that $[f]:[X] \rightarrow [Y]$ induces a strong morphism of $k$-analytic Berkovich spaces $([X], \mt A_X|_f, \tau_X|_f) \rightarrow ([Y],\mt A_Y, \tau_Y)$:\\ 
Note that property (i) of definition \ref{def. strong morphisms of k-analytic spaces} is satisfied by the definition of $ \tau_X|_f$. So let $[U]\in \tau_X|_f$ and $[U'] \in \tau_Y$ such that $[f]([U]) \subseteq [U']$. As argued in the construction above, we have $f(U) \subseteq U'$ by \ref{lemma: connection inclusion of subsets on separated quotient}. Hence the bounded $k$-algebra homomorphism $f_{U'}^b:\mt O_Y(U')\rightarrow \mt O_X (f^{-1}(U'))$ composed with the restriction map $\text{res}_{f^{-1}(U'}^U$ induces a bounded $k$-algebra homomorphism $\phi:\mt O_Y(U')=A'(U')\rightarrow A(U)= \mt O_X(U)$. Again as in construction \ref{construction: functor objects} we see that $\phi$ corresponds via the functor $(A,A^o)\mapsto \Spa(A,A^o)$ to the morphism of adic spaces $f|_U :U \rightarrow U'$. Hence the obtained morphism of $k$-affinoid spaces $\mt M(A_U)\rightarrow \mt M(A'_{U'})$ corresponds to $[f|_U]$ via the diagram from \ref{exmp: separation map adic spectra; adic spaces }. But $[f|_U]$ is equal to $[f]|_{[U]}$ which shows that we constructed a strong morphism.\\ 
Finally, we define the morphism of Hausdorff strictly $k$-analytic Berkovich spaces associated to $(f,f^b)$ to be the equivalence class of the diagram 
\begin{gather*}		
([X],\mt A_X, \tau_X)\xlongleftarrow{\id_{[X]}} ([X], \mt A_X|_f,\tau_X|_f) \xlongrightarrow{[f]} ([Y],\mt A_Y, \tau_Y)
\end{gather*}
(with respect to the equivalence relation defined in \ref{def. category k-affinoid spaces}).
Note that this indeed is a morphism $([X],\mt A_X, \tau_X)\rightarrow ([Y],\mt A_Y, \tau_Y)$ of Hausdorff strictly $k$-analytic Berkovich spaces (remember that the category of (strictly) $k$-analytic Berkovich spaces is a category of fractions, cf. \ref{def. category of k-analytic spaces}).
}
\end{const}

\begin{theo}\label{theo: functor well defined}
There is a functor from the category of taut adic spaces that are locally of finite type over $k$ to the category of strictly $k$-analytic Hausdorff spaces, sending $(X, \mt O_X , (v_x)_{x\in X})$ to $([X],\mt A_X, \tau_X)$ as in \ref{construction: functor objects} and a morphism of adic spaces $(f,f^b):(X, \mt O_X , (v_x)_{x\in X})\rightarrow (Y, \mt O_Y , (v_y)_{y\in Y})$ to $([X],\mt A_X, \tau_X)\xlongleftarrow{\id_{[X]}} ([X], \mt A_X|_f,\tau_X|_f) \xlongrightarrow{[f]} ([Y],\mt A_Y, \tau_Y)$ as in \ref{const: functor morphisms}. We will denote this functor by $\mt F$.
\end{theo}
\begin{proof}
It is left to show that $\mt F$ respects composition of morphisms. So consider the following diagram of morphisms of taut adic spaces that are locally of finite type over $k$:
\begin{gather*}		
(X, \mt O_X , (v_x)_{x\in X})\xrightarrow{(f,f^b)} (Y, \mt O_Y , (v_y)_{y\in Y}) \xrightarrow{(g,g^b)} (Z, \mt O_Z, (v_z)_{z \in Z}).
\end{gather*}
By applying $\mt F$, we get:
\begin{gather*}		
([X],\mt A_X, \tau_X)\xrightarrow{\mt F((f,f^b))}([Y],\mt A_Y, \tau_Y)\xrightarrow{\mt F((g,g^b))}([Z],\mt A_Z, \tau_Z).
\end{gather*}
As above let $\delta_X$, $\delta_Y$ and $\delta_Z$ be the sets containing the open affinoid subsets of $X$, resp. $Y$, resp. $Z$. 
Consider 
\begin{gather*}		
\delta_X'':= \lbrace U \in \delta_X \ | \ \exists \ V \in \delta_Y \ \text{and}  \ W \in \delta_Z \ \text{such that} \ f(U)\subseteq V \ \text{and} \ g(V)\subseteq W \rbrace \\ 
=\lbrace U \in \delta_X |_f \ ; \ \exists \ V \in \delta_Y |_g \ \text{with} \ f(U)\subseteq V \rbrace
\end{gather*}
Then the elements of $\delta_X''$ still form a basis of the topology of $X$ and therefore $\tau_X'':=\text{sep}_X(\delta_X'')$ forms a net on $[X]$ by \ref{prop: when net on separated quotient}. Let $\mt A_X''$ denote the restriction of the $k$-affinoid atlas $\mt A_X$ to $\tau_X''$. We obtain a Hausdorff strictly $k$-analytic Berkovich space $([X],\mt A_X'', \tau_X'')$. Using the definition of $\delta_X''$, we can in an natural way complement $([X], \mt A_X|_f ,\tau_X|_f) \rightarrow ([Y],\mt A_Y, \tau_Y) \leftarrow ([Y],\mt A_Y|_g, \tau_Y|_g) $ to a commutative diagram of strong morphisms
\begin{gather*}		
\begin{xy} 
  \xymatrix{
      ([X], \mt A_X|_f ,\tau_X|_f) \ar[r]    &   ([Y],\mt A_Y, \tau_Y)  \\
      ([X],\mt A_X'', \tau_X '') \ar[r] \ar[u]             &   ([Y],\mt A_Y|_g, \tau_Y|_g) \ar[u]
  }
\end{xy}
\end{gather*}
 in which the vertical arrows are quasi-isomorphisms.\\
Hence $([X],\mt A_X, \tau_X)\xlongleftarrow{\id_{[X]}} ([X], \mt A_X'',\tau_X'') \xlongrightarrow{[g\circ f]} ([Z],\mt A_Z, \tau_Z)$ is a representative of $\mt F((g,g^b)) \circ \mt F((f,f^b))$ (remember the definition of composition of morphisms \ref{def. category of k-analytic spaces}). But since $\tau_X''$ is a subset of $\tau_X|_{g\circ f}$, this diagram is equivalent to $([X],\mt A_X, \tau_X)\xlongleftarrow{\id_{[X]}} ([X], \mt A_X|_{g\circ f},\tau_X|_{g \circ f}) \xlongrightarrow{[g\circ f]} ([Z],\mt A_Z, \tau_Z)$ as desired.
\end{proof}

For the constructions above it was essential to know that a basis of the topology on the adic side leads to a net on the Berkovich level via the separation map (cf. \ref{prop: when net on separated quotient}). The opposite direction is not true. Consider a $k$-affinoid space $X= \mt M(A)$ with trivial net $\tau=\lbrace X \rbrace$. In this case $\lbrace U \subseteq Spa(A,A^o) \ | \ U \ \text{open and affinoid}, [U] \in \tau \rbrace$ is in general not a basis of the topology of $\Spa(A,A^o)$. However, this question plays an important role for constructions that go in the other direction, namely when we show that $\mt F$ is essentially surjective and full. Luckily, we have the following result instead:

\begin{lem}\label{lem. when net gives basis of the topology}
Let $(X, \mt O_X, (v_x)_{x \in X})$ be a taut adic space that is locally of finite type over $k$ and $([X], \mt A _X, \tau_X)$ be its image under $\mt F$. Moreover, let $\tau'\subseteq \tau_X$ also be a net on $[X]$. Then $W=\lbrace U \subseteq X \ | \ U \text{open and affine} \ , [U] \in \overline{\tau'} \rbrace $ forms a basis of the topology of $X$ (recall that $\overline{\tau'}$ is by definition the set containing all affinoid domains of all elements of $\tau'$; cf. \ref{rem. refinement of nets in Berkovich spaces: include all affinoids}). 
\end{lem}
\begin{proof}
Since the elements of $\tau'$ form a net of $[X]$, they cover $[X]$. Hence the sets in $W':=\lbrace U \subseteq X \ | \ U  \ \text{open and affine} \ , [U] \in \tau' \rbrace $ form an open cover of $X$ because $X$ is reflexive (cf. \ref{prop: reflexive spaces alternative description}). Since $\overline{\tau'}$ contains all rational domains of all elements of $\tau'$, $W$ contains all rational subsets of all elements of $W'$. This shows the claim.
\end{proof}

\begin{const}\label{construction: quasi-inverse functor}($\mt F$ is essentially surjective)
\emph{Let $X=(X,\mt A ,\tau)$ be a Hausdorff strictly $k$-analytic Berkovich space. In this paragraph, we construct a taut adic space $(X^{ad},\mt O_X, (v_x)_{x \in X^{ad}})$ that is locally of finite type over $k$ in such a way that $\mt F(X^{ad})$ is isomorphic to $X$. In other words, we show that $\mt F$ is essentially surjective. The proof is just a glueing argument mainly basing on the fact that affinoid domain embeddings induce open immersions of affinoid adic spaces (cf. \ref{prop. affinoid domain embedding open immersion of adic spaces}). \\
We consider the Hausdorff strictly $k$-analytic Berkovich space $(X,\overline{\mt A}, \overline{\tau})$  which is isomorphic to $X$ and whose net consists of all affinoid domains of all elements of $\tau$ (cf. \ref{rem. refinement of nets in Berkovich spaces: include all affinoids}).\\
As in example \ref{rem. properties of valued spaces for Berkovich}, we consider the taut reflexive valuative space  
\begin{gather*}		
X^{ad}:= \lim\limits_{\underset{S \in \overline{\tau}}\longrightarrow} X_S
\end{gather*}
 where $X_S:= Spa(A_S,A_S^o)$. Moreover, for $S \in \overline{\tau}$ we also have open topological immersions $\pi_S: X_S \rightarrow X^{ad}$. Under the homeomorphism $[X^{ad}] \cong X$, the separated quotient $[\pi_S(X_S)]$ corresponds to $S \subseteq X$ and hence
$(\pi_S(X_S))_{S \in \overline{\tau}}$ is a basis of the topology of $X^{ad}$ by \ref{lem. when net gives basis of the topology} (note that $X^{ad}$ is taut by \ref{prop. properties of glued valued space}).\\
Now we want to define a sheaf of complete topological rings on $X^{ad}$:\\
$\pi_S(X_S) \cong X_S = \Spa(A_S,A_S^o)$ provide sheaves of complete topological rings on an open covering of $X^{ad}$ which we want to glue together.  By \ref{prop. affinoid domain embedding open immersion of adic spaces} those sheaves coincide on sets of the form $\pi_S(X_S)$ where $S \in \overline{ \tau}$. Since those sets form a basis of the topology of $X^{ad}$, they coincide on intersections and hence can be glued together to obtain a sheaf $\mt O _{X^{ad}}$ on $X^{ad}$ such that $\mt O_{X^{ad}}|_{\pi_S(X_S)} \cong \mt O_{X_S}$ for all $S \in \overline{\tau}$.\\
Let $x\in X^{ad}$ with $x \in \im(\pi_S)$ for some $S \in \tau$ we have $\mt O _{X^{ad},x} \cong \mt O_{{X_S},\pi_S^{-1}(x)}$ and hence $(X^{ad},\mt O _{X^{ad}})$ is a locally ringed space. Moreover, for such an $x \in X^{ad}$ we define $v_x$ to be the valuation on $\mt O _{X^{ad},x}$ corresponding to the valuation on $\mt O_{{X_s},\pi_S^{-1}(x)}$ obtained from the adic space structure on $X_S$. Note that this is well defined: In fact, if $x \in \im(\pi_S')$ for another $S' \in \tau$, we can find $S'' \in \overline{\tau}$ with $x \in \im(\pi_{S''}) \subseteq (\im(\pi_S) \cap \im(\pi_{S'}))$. But this means $S'' \subseteq S \cap S'$ and hence we have open immersions of adic spaces $\Spa(A_{S''},A_{S''}^o)\rightarrow \Spa(A_{S'},A_{S'}^o)$ and $\Spa(A_{S''},A_{S''}^o) \rightarrow \Spa(A_S,A_S^o)$ respectively (cf. \ref{prop. affinoid domain embedding open immersion of adic spaces}).\\
Since $X_S$ is an adic space, which is locally of finite type over $k$ for any $S \in \tau$, we obtain a taut adic space $X^{ad}=(X^{ad},\mt O_{X^{ad}}, (v_x)_{x \in X^{ad}})$ which also is locally of finite type over $k$.\\
We finally show that $(X,\mt A ,\tau )$ is isomorphic to 
\begin{gather*}		
([X^{ad}],\mt A', \tau ') := \mt F((X^{ad},\mt O_{X^{ad}}, (v_x)_{x \in X^{ad}}))
\end{gather*}
 as a Hausdorff strictly $k$-analytic Berkovich space:\\
$[X^{ad}]$ is homeomorphic to $X$ by \ref{rem. on glueing valued spaces}, so we do not distinguish between those two topological spaces. $S \in \tau$ corresponds to an open affinoid subset $\pi_S(\Spa(A_S,A_S^o))$ of $X^{ad}$ which can in turn be identified with $S$ in its separated quotient.  Hence $S \in \tau '$ and $ \mt A(S) = \mt A'(S)$ by construction (cf. \ref{const: functor morphisms}). But this provides a quasi-isomorphism $X=(X,\mt A ,\tau)\rightarrow ([X^{ad}],\mt A', \tau')$ and therefore the claim follows.
}
\end{const}

Given a morphism $\Phi:\mt F(X) \rightarrow \mt F(Y)$ of Hausdorff strictly $k$-analytic Berkovich spaces (for taut adic spaces $X$ and $Y$ that are locally of finite type over $k$), we want to use the universal property of the $(\cdot)^{ad}$ construction to obtain a morphism of adic spaces $\mt F(X)^{ad} \rightarrow \mt F(Y)^{ad}$. To show that this morphism provides a pre-image of $\Phi$ under $\mt F$, we have to check that $\mt F(X)^{ad} $ and $\mt F(Y)^{ad}$ are connected to $X$ and $Y$ respectively. For this purpose note the following remark.

\begin{rem}\label{rem. ad is quasi inverse functor}
\emph{\begin{enumerate}
\item
The $(\cdot)^{ad}$ construction of \ref{construction: quasi-inverse functor} provides a quasi-inverse functor to $\mt F$ on the object level. Indeed: Let $X$ be a taut adic space that is locally of finite type over $k$ and $\mt F(X)=([X], \mt A_X ,\tau_X)$. Let $\delta_X$ denote the collection of all open affinoid subsets of $X$. Then by construction we have
\begin{gather*}		
\mt F(X)^{ad} = \lim\limits_{\underset{U \in \delta_X}\longrightarrow} \Spa(\mt O_X(U), \mt {O_X}(U)^o).
\end{gather*}
We see that $\mt F(X)^{ad}$ and $X$ are isomorphic since for every $U \in \delta_X$, we have an open immersion of adic spaces $U \hookrightarrow \mt F(X)^{ad} $ that glue to an isomorphism $X \rightarrow \mt F(X)^{ad}$. \\
\item In \ref{const: functor morphisms} we had to coarsen the net of a $k$-analytic space. For the following constructions we need a similar procedure on the level of adic spaces.
Let $\delta_X'$ be a collection of open affinoid subsets of $X$ that form a covering of $X$. Then 
\begin{gather*}		
X \cong  \lim\limits_{\underset{U \in \delta_X}\longrightarrow} \Spa(\mt O_X(U), \mt {O_X}(U)^o) =  \lim\limits_{\underset{U \in \delta_X'}\longrightarrow} \Spa(\mt O_X(U), \mt {O_X}(U)^o).
\end{gather*}
\end{enumerate}
}
\end{rem}

\begin{const}\label{const. inverse function on mapping sets}($\mt F$ is full)
\emph{Let $X$ and $Y$ be taut adic spaces that are locally of finite type over $k$. In this section we want to show that any morphisms of Hausdorff strictly $k$-analytic Berkovich spaces 
\begin{gather*}		
  \Phi : \mt F(X)=([X],\mathcal{A}_X,\tau_X) \rightarrow  ([Y],\mathcal{A}_Y,\tau_Y)= \mt F(Y)
\end{gather*} 
is the image of a morphism of adic spaces $X \rightarrow Y$ under the functor $\mt F$. In other words, we show that $\mt F$ is full.\\
We may assume that the span
\begin{gather*}		
([X],\mathcal{A}_X,\tau_X) \xlongleftarrow{\id_{[X]}}([X],\mathcal{A}_X'',\tau_X'')  \xlongrightarrow{\varphi}  ([Y],\mathcal{A}_Y,\tau_Y)
\end{gather*}
is a representative of $\Phi$,
where $\tau''_X$ is assumed to be the maximal subset of $\tau_X$ such that $\varphi$ is a strong morphism. In particular, this means that $\tau_X''$ contains all $[U]$ where $U\subseteq X$ open and affinoid, such that there exists a $V\subseteq Y$ open and affinoid with $\varphi([U]) \subseteq [V]$.
Moreover, $\varphi$ induces a strong morphism
\begin{gather*}		
 \overline{\varphi}:([X],\overline{\mathcal{A}_X''},\overline{\tau_X''})  \rightarrow  ([Y],\overline{\mathcal{A}_Y},\overline{\tau_Y})
  \end{gather*} (cf. \ref{rem. refinement of nets in Berkovich spaces: include all affinoids} iii)). For brevity we write $A_U$ for $\overline{\mt A_X''}(U)$ with $U \in \overline{\tau_X''}$ and $A'_V$ for $\overline{\mt A_Y}(V)$ with $V \in \overline{\tau_Y}$.\\
Note that by \ref{lem. when net gives basis of the topology} the sets 
\begin{gather*}		
\delta_X'':=\lbrace U \subseteq X \ | \ \text{open and affinoid, such that} \ [U] \in \overline{\tau_X''} \rbrace  \\
\delta_Y'':=\lbrace V \subseteq Y \ | \ \text{open and affinoid, such that} \ [V] \in \overline{\tau_Y}. \rbrace  
\end{gather*}
form basis of the topologies of $X$ and $Y$ respectively.
By \ref{rem. ad is quasi inverse functor} we have 
\begin{gather*}
X \cong \lim\limits_{\underset{U \in \delta_X''}\longrightarrow} \Spa(\mt O_X(U), \mt {O_X}(U)^o)\\  
Y \cong \lim\limits_{\underset{V \in \delta_Y''}\longrightarrow} \Spa(\mt O_Y(V), \mt {O_Y}(V)^o).
\end{gather*}
For all $U \in \delta_X''$ there exists a $V \in \delta_Y''$ and a morphisms of strictly $k$-affinoid spaces 
\begin{gather*}		
\mt M(A_{[U]})\rightarrow \mt M(A'_{[V]})
\end{gather*}
which is induced by $\overline{\varphi}$. This provides a morphisms of affinoid adic spaces
\begin{gather*}		
\Spa(A_U,A_U^o) \rightarrow \Spa(A'_V,{A'_V}^o).
\end{gather*}
Hence we obtain a continuous map $f: X \rightarrow Y$ by the universal property of $X$ as a colimit. 
Now we want to define an appropriate morphism of sheaves $f^b: \mt O_Y \rightarrow f^{*} \mt O_X$ such that $\mt F((f,f^b))= \Phi$.\\
As we have seen above for  $U \in \delta_X''$ and $V \in \delta_Y''$ such that $f(U)\subseteq V$ we have a bounded homomorphism of complete rings $\mt O_Y(V) \rightarrow \mt O_X(U)$. $\delta_X''$ and $\delta_Y''$ form basis of the topologies of $X$ and $Y$ respectively. Therefore those ring homomorphisms define $f_W^b$ for an arbitrary open subset $W$ of $Y$ by using suitable restriction maps and the usual colimes construction. Therefore we obtain a morphism of ringed spaces $(f,f^b):X \rightarrow Y$.\\
We still have to check that this morphism respects the valuations on the stalks. But this is clear since $(f,f^b)$ is locally given by morphisms of affinoid adic spaces.
To prove that $\mt F$ is full, we show that 
\begin{gather*}		
([X],\mathcal{A}_X,\tau_X) \xlongleftarrow{\id_{[X]}}([X],\mathcal{A}_X|_f,\tau_X|_f)  \xlongrightarrow{[f]}  ([Y],\mathcal{A}_Y,\tau_Y) \\ \cong  \\([X],\mathcal{A}_X,\tau_X) \xlongleftarrow{\id_{[X]}}([X],\mathcal{A}_X'',\tau_X'')  \xlongrightarrow{\varphi}  ([Y],\mathcal{A}_Y,\tau_Y),
\end{gather*}
where the first span is a representative of $\mt F((f,f^b))$.
By the construction of $f$ it is clear that $[f]$ and $\varphi$ coincide set theoretically, since $[f]|_{[U]}$ and $\varphi|_{[U]}$ both correspond to $\mt M(A_{[U]}) \rightarrow \mt M(A'_{[V]})$ for $U \in \delta_X''$ and suitable $V$. For $U \subseteq X$, $V\subseteq Y$ open and affine with $f(U)\subseteq V$ we have $[f]([U])\subseteq [V]$ and hence $\tau_X|_f \subseteq \tau_X''$ by our particular choice of $\tau_X''$. So we get a quasi-isomorphism $([X],\mathcal{A}_X|_f,\tau_X|_f) \rightarrow ([X],\mathcal{A}_X',\tau_X')$. Therefore the spans considered above are equivalent.
}
\end{const}

\begin{const}\label{const: F  faithful}($\mt F$ is faithful)
\emph{ Let $(f,f^b),(g,g^b):X\rightarrow Y$ be morphisms of taut adic spaces locally of finite type over $k$, such that the spans 
\begin{gather*}
\mt F(X)=([X],\mathcal{A}_X,\tau_X) \xlongleftarrow{\id_{[X],f}}([X],\mathcal{A}_X|_f,\tau_X|_f)  \xlongrightarrow{[f]}  ([Y],\mathcal{A}_Y,\tau_Y) = \mt F(Y) \\
\cong \\
([X],\mathcal{A}_X,\tau_X) \xlongleftarrow{\id_{[X],g}}([X],\mathcal{A}_X|_g,\tau_X|_g)  \xlongrightarrow{[g]}  ([Y],\mathcal{A}_Y,\tau_Y)
\end{gather*}
are equivalent.
By definition of the equivalence relation there exists a Hausdorff strictly $k$-analytic Berkovich space $\overline{X}$, and strong morphisms $h,l$ such that $\id_{[X],f}\circ h$, $\id_{[X],g}\circ l$ are quasi-isomorphisms and the diagram
\begin{gather*}
\begin{xy}
  \xymatrix{
		& ([X],\mathcal{A}_X|_f,\tau_X|_f)  \ar[ld]_{\id_{[X],f}} \ar[rd]^{[f]}  \\
		([X],\mathcal{A}_X,\tau_X) & \overline{X}\ar[u]_h \ar[d]_l & ([Y],\mathcal{A}_Y,\tau_Y)\\
		& ([X],\mathcal{A}_X|_g,\tau_X|_g)  \ar[lu]^{\id_{[X],g}} \ar[ru]_{[g]}\\
 }
\end{xy}
\end{gather*}
commutes. This means that $[f]=[g]$ as continuous maps and hence $f=g$ as continuous maps by reflexivity of $X$ (cf. \ref{prop: maps of reflexive spaces}).
It is left to show that for all open subsets $V \subseteq Y$ the map $f_V^b: \mt O_Y(V) \rightarrow \mt O_X(f^{-1}(V))$ is uniquely determined by $[f]=[g]$ (and hence is equal to $g_V^b$). Let $V \subseteq Y$ be open, $W' \subseteq V$ be an open affinoid subset and $U'$ open and affinoid in $X$, such that $f(U')\subseteq W'$. The morphism $f_{W',U'}^b :\mt O_Y(W') \rightarrow \mt O_X(U')$ we obtain in this way corresponds by construction of $\mt F ((f,f^b))$ to $[f]|_{[U']}$ and therefore just depends on $[f]$. We put $f_{V,U'}^b:=f_{V,W'}^b \circ res^{W'}_{V}$. Then we have 
\begin{gather*}		
f_V^b=\lim_{\longleftarrow} f_{V,U}^b \	,
\end{gather*}
where the limit is taken over all $U\subseteq X$ open and affinoid such that there is an open and affinoid $W \subseteq V$ with $f(U) \subseteq W$. This shows the claim.
}
\end{const}

Combining all constructions in this paragraph leads to the final theorem of this paper:

\begin{theo}\label{theo. final theorem}
There is an equivalence of categories:
\begin{gather*}		
\lbrace \text{taut adic spaces that are locally of finite type over }k \rbrace \\	
\cong	\\
\lbrace \text{Hausdorff strictly} \ k\text{-analytic Berkovich spaces} \rbrace
\end{gather*}
sending $(X,\mt O_X, (v_x)_{x\in X})$ to $([X],\mt A, \tau)$ as in construction \ref{construction: functor objects}. 
\end{theo}

\end{document}